\documentclass[12pt, a4paper, reqno]{amsart}
\usepackage{amscd, amssymb}
\usepackage{amsthm}
\usepackage{tikz}
\usetikzlibrary{arrows,positioning} 
\usepackage{mathrsfs}
\usepackage{bbm}
\usepackage{ stmaryrd }

\usetikzlibrary{positioning,calc,arrows.meta,shapes.geometric,fit}

\usepackage{hyperref}
\hypersetup{
	colorlinks=true,
	linkcolor=blue,
	filecolor=magenta,  
	urlcolor=cyan,
}

\usepackage{hyperref}
\usepackage{cleveref}

\def\Hess{\operatorname{Hess}}
\def\gr{\operatorname{gr}}

\def\Res{\operatorname{Res}}

\def\ker{\operatorname{ker}}

\def\im{\operatorname{Im}}

\def\dim{\operatorname{dim}}
\def\sdim{\operatorname{sdim}}

\def\Ber{\operatorname{Ber}}
\def\Ind{\operatorname{Ind}}

\def\Hom{\operatorname{Hom}}

\def\ad{\operatorname{ad}}

\def\C{\mathbb{C}}

\def\R{\mathbb{R}}
\def\N{\mathbb{N}}
\def\Z{\mathbb{Z}}

\def\TT{\mathcal{T}}

\def\KK{\mathcal{K}}

\def\UU{\mathcal{U}}
\def\MM{\mathcal{M}}

\def\FF{\mathcal{F}}

\def\JJ{\mathcal{J}}

\def\a{\mathfrak{a}}
\def\b{\mathfrak{b}}
\def\c{\mathfrak{c}}
\def\fd{\mathfrak{\fd}}
\def\m{\mathfrak{m}}
\def\n{\mathfrak{n}}
\def\p{\mathfrak{p}}
\def\q{\mathfrak{q}}
\def\g{\mathfrak{g}}
\def\h{\mathfrak{h}}

\def\k{\mathfrak{k}}
\def\l{\mathfrak{l}}
\def\s{\mathfrak{s}}
\def\o{\mathfrak{o}}
\def\u{\mathfrak{u}}

\def\ol{\overline}

\def\Ext{\text{Ext}}

\def\sub{\subseteq}

\def\t{\tilde}

\newtheorem{thm}{Theorem}[section]
\newtheorem{cor}[thm]{Corollary}
\newtheorem{conj}[thm]{Conjecture}
\newtheorem{lemma}[thm]{Lemma}
\newtheorem{prop}[thm]{Proposition}

\theoremstyle{definition}
\newtheorem{definition}[thm]{Definition}
\theoremstyle{remark}
\newtheorem{remark}[thm]{Remark}

\numberwithin{equation}{section}

\usepackage{relsize}
\usepackage{fancyhdr}
\usepackage[all,cmtip]{xy}

\begin{document}
	\title{Splitting quasireductive supergroups and volumes of supergrassmannians}
	
	\author{ Vera Serganova, Alexander Sherman }
	\pagestyle{plain}
	\maketitle

        {\it To Yu. I. Manin with admiration and gratitude}

	\begin{abstract}We introduce the notion of splitting subgroups of quasireducitve supergroups, and explain their significance.  For $GL(m|n)$, $Q(n)$, and defect one basic classical supergroups, we give explicit splitting subgroups.  We further prove they are minimal up to conjugacy, except in the $GL(m|n)$ case where it remains a conjecture.  A key tool in the proof is the computation of the volumes of complex supergrassmannians, which is of interest in its own right. \end{abstract}
	
	\section{Introduction} 	
	
	In this paper we deal with representations of quasireductive supergroups $G$ over $\mathbb{C}$, meaning that $G_0$, the underlying group of $G$, is reductive. Many important supergroups (including $GL(m|n)$, $SOSp(m|2n)$, $Q(n)$,..) are quasireductive, and although their representation theory is similar to that of reductive groups, semisimplicity of $\operatorname{Rep}G$ is very rare.  On the other hand, the property of being quasireductive is equivalent to $\operatorname{Rep}G$ being Frobenius, a property shared by representations of finite groups in positive characteristic. 
	
	We say a quasireductive subgroup $K\sub G$ is \emph{splitting} if the trivial $G$-submodule of $\C[G/K]$ spanned by the constant function 1 splits off.   In other words, there is a $G$-submodule $V\sub \C[G/K]$ such that we have an isomorphism of $G$-modules $\C[G/K]\cong \C\langle 1\rangle\oplus V$. \color{black}  Here $\C[G/K]$ denotes the superalgebra of polynomial functions on $G/K$.  In modular representation theory, the analogous notion is very natural.  Indeed, if $G$ is a finite group and $\Bbbk$ is a field of characteristic $p>0$, then the $G$-submodule of constant functions in $\Bbbk[G/K]$ splits off if and only if $|G|/|K|$ is prime to $p$.  In particular, the minimal `splitting' subgroups in the modular setting would be the Sylow $p$-subgroups of $G$, which are all conjugate to one another.  In modular representation theory in characteristic $p$, one may often reduce questions about $G$-modules to analogous questions about modules over a (normalizer of a) Sylow $p$-subgroup (see \cite{B} for more on modular representation theory).

	Heuristically speaking, a splitting subgroup $K$ of $G$ has representation theory that is at least as complicated as that of $G$, despite being a smaller group.  For example, if the trivial subgroup of $G$ is splitting, then $G$ has semisimple representation theory.  Further, projectivity of a $G$-module can always be checked by restricting to a splitting subgroup.

	\subsection{Main theorem} In this paper we initiate a study of minimal splitting subgroups $K$ of distinguished quasireductive supergroups $G$; or if not minimal, then at least conjecturally so.  The supergroups we consider are $GL(m|n), Q(n),$ and all other supergroups $G$ such that $\operatorname{Lie}G$ is either $\p\s\l(n|n)$ or simple basic classical of defect 1.  In particular we show the following:
		\begin{thm}\label{theorem_intro_main}
		The following subgroups are splitting.
		\begin{enumerate}
			\item 
			\[
			SL(1|1)^{\min(m,n)}\sub GL(m|n), \ \ \ P\left(SL(1|1)^{n}\right)\sub PSL(n|n);
			\]
			\item 
			\[
			Q(2)^n\sub Q(2n), \ \ \ \ Q(2)^n\times Q(1)\sub Q(2n+1);
			\]
			\item for any $G$ with $\operatorname{Lie}G$ a defect 1 basic classical Lie superalgebra:
			\[
			SL(1|1)\sub G.
			\]
		\end{enumerate}
	Further, in cases (2) and (3), these subgroups are minimal splitting subgroups, and are unique up to conjugacy with this property.
	\end{thm}
	In the above, the subgroups $SL(1|1)^{\operatorname{def}\g}$, where $\operatorname{def}\g$ is the defect of $\g$ (see Section \ref{section_def_subgroups}), are those that arise from the root subgroup of $(\operatorname{def}\g)$-many isotropic, mutually orthogonal odd roots of $\g$.  In the case of $Q(n)$ we take the natural root subgroup of the specified type (in fact it is a Levi subgroup).
	
	The proof of (3) can be done using straightforward techniques in the algebraic category, in particular by studying the space of functions on supersymmetric spaces for the defect one supergroups.  However for the infinite series $GL(m|n)$ and $Q(n)$ we need techniques from differential geometry.
	
	\subsection{Connections to support varieties} In \cite{BoKN1}, \cite{BoKN2}, and \cite{GGNW}, cohomological support varieties are studied for quasireductive Lie superalgebras $\g$.  There, they prove that for basic Lie superalgebras and $\q(n)$, there are detecting subalgebras $\mathfrak{e},\mathfrak{f}\sub\g$ on which one can compute the support variety of a given $\g$-module $M$ after restricting to $\mathfrak{e}$ or $\mathfrak{f}$.  In the basic classical cases, their subalgebras $\mathfrak{f}$ agree with the Lie superalgebras of our splitting subgroups above.  For $\q(n)$, they have $\mathfrak{e}=\mathfrak{f}=\q(1)^n$, so their detecting subalgebras are smaller.  
	
	In a forthcoming work, we will study rank varieties determined by the Duflo-Serganova functor for homological elements, i.e. $\g_{\ol{1}}^{hom}$ (see Section \ref{section_odd_semisimple}).  Our splitting subgroups will play an important role in these rank varieties; it will be interesting to understand the connections of these rank varieties to the ones defined in \cite{BoKN1}.
	
	\subsection{Volumes of supergrassmannians}  In order to prove parts (1) and (2) of Theorem \ref{theorem_intro_main}, we use a unitary trick.  First of all, by Lemma \ref{lem_split_quotient}, we may reduce (1) to showing that $GL(1|1)^{\min(m,n)}\sub GL(m|n)$ is splitting.  Also, the case of $PSL(n|n)$ may be reduced to that of $GL(n|n)$, so we put this case aside in this introductory discussion.
	
	Now both $GL(m|n)$ and $Q(n)$ admit real compact forms $U(m|n)$ and $UQ(n|n)$.  If we consider the natural embeddings 
	\[
	U(1|1)^{\min(m,n)}\sub U(m|n), \ \ \ \ UQ(2)^{n}(\times UQ(1))\sub UQ(2n(+1)),
	\]
	then their complexifications are the subgroups we are interested in showing are splitting.  On the other hand, their quotients are compact supermanifolds admitting invariant volume forms; thus if we can show their volumes are nonzero, we may obtain our desired splitting via the integration map.

	Instead of computing these volumes directly, we answer a separate question, which is sufficient for our purposes via an inductive argument for splitting subgroups. Consider the supergrassmannian $Gr(r|s,m|n)$ or $(r|s)$-planes in $\C^{m|n}$.  This is a homogeneous space for $U(m|n)$ with stabilizer of a point given by $U(r|s)\times U(m-r|n-s)$.   The study of supergrassmannians was initiated in the early 80s by Manin and his students, see \cite{M} and \cite{PS}.

	We prove that:
	
	\begin{thm}\label{theorem_intro_volume_gl}
		Let $\omega$ be an invariant volume form on $\MM:=Gr(r|s,m|n)$; then
		\[
		\int_{\MM}\omega\neq 0
		\]
		if and only if $\sdim\MM\geq0$.  
	\end{thm}
	An explicit formula for this volume is given in Corollary \ref{formula}.  We note that while our formula does not agree with conjecture given in \cite{V}, the two formulas are consistent on when the volume is nonzero.

	Now using the unitary trick, Theorem \ref{theorem_intro_volume_gl} implies that $GL(r|s)\times GL(m-r|n-s)\sub GL(m|n)$ is splitting if and only if $(r-s)((m-r)-(n-s))\geq0$.  With the help of the transitivity of splitting subgroups (see Corollary \ref{cor_splitting_subgroup}), one can obtain part (1) of Theorem \ref{theorem_intro_main}.

	For the $Q(n)$ case, we take $\C^{n|n}$ with a fixed odd linear automorphism $F:\C^{n|n}\to\C^{n|n}$, and consider the space $QGr(r,n)$, which is the supergrassmannian of $(r|r)$-dimensional, $F$-stable subspaces of $\C^{n|n}$.  This is a homogeneous space for $UQ(n)$ with stabilizer $UQ(r)\times UQ(n-r)$. 
	\begin{thm}\label{theorem_intro_volume_q}
		Let $\omega$ be an invariant volume form on $\MM:=QGr(r,n)$; then 
		\[
		\int_{\MM}\omega\neq0
		\]
		if and only if $r(n-r)$ is even.
	\end{thm}
	From this we can again use the unitary trick and transitivity to prove part (2) of Theorem \ref{theorem_intro_main}.
	
	\subsection{Localization of integrals} In order to prove Theorems \ref{theorem_intro_volume_gl} and \ref{theorem_intro_volume_q}, we use the results of \cite{SZ}, where a localization theorem was proven for integrals on supermanifolds.  We use the following version of it.  
	
	Let $\MM=Gr(r|r,n|n)$ or $QGr(r,n)$, and $\UU=U(n|n)$ or $UQ(n)$ in the respective cases; then there exists an element $Q\in(\operatorname{Lie}\UU)_{\ol{1}}$ such that the zeroes of $Q$ are isolated and nondegenerate; in other words if $p$ is a zero of $Q$, then the natural endomorphism $Q:T_p\MM\to T_p\MM$ is an isomorphism. 
	
	In this case, the result of \cite{SZ} implies that the integral $\int_{\MM}\omega$ localizes to $Z(Q)$, the zero set of $Q$.  In particular we have that
	\[
	\int_{\MM}\omega=(2\pi)^{\dim\MM_0}\sum\limits_{p\in Z(Q)}\alpha(T_p\MM,\omega_p),
	\]
	where $\alpha(T_p\MM,\omega_p)\in\C$ can be computed in terms of the form $\omega_p$ and the action of $Q$ on $T_p\MM$ (see Section \ref{section_5_alpha} for the precise definition).  In our situation we are able to explicitly compute $\alpha(T_p\MM,\omega_p)$ to obtain the values of our integrals and show they are nonzero.  To deal with the general case of $Gr(r|s,m|n)$, either the localization theorem applies similarly to the above, or we can apply a trick of realizing the two step flag variety $Fl(r_1|s_1,r_2|s_2,m|n)$ as a fiber bundle in two differ ways, allowing us to obtain relationships between volumes of different Grassmannians.

	\subsection{Summary of sections} Section 2 introduces splitting subgroups and proves general properties they satisfy.  Section 3 defines contragredient supergroups and their defect subgroups, which are the natural candidates for minimal splitting subgroups.  We prove they are minimal in the defect one case, and also discuss what happens for $G=Q(n)$.  In Section 4 we prove defect subgroups are splitting for all defect one contragredient supergroups via transitivity and the use of symmetric spaces.  Section 5 computes the volume of the supergrassmannians for $GL(m|n)$ and $Q(n)$ via the localization technique of \cite{SZ}.  Finally Section 6 applies the volume computations to proving Theorem \ref{theorem_intro_main}.
	
	\subsection{Acknowledgements} The authors are grateful to A. Schwarz and T. Voronov for numerous helpful discussions.  The first author was supported in part by NSF grant 2001191.  The second author was supported in part by ISF grant 711/18 and NSF-BSF grant 2019694.
	
	\section{Splitting subgroups}
	\subsection{Preliminaries and notation} We work throughout over the field $\C$ of complex numbers.  For a super vector space $V$, we write $V=V_{\ol{0}}\oplus V_{\ol{1}}$ for its $\Z_2$-gradation.  We will write $\MM$ for a supervariety, and let $\C[\MM]$ denote the algebra of global functions on $\MM$.  If $\MM$ is smooth, we set $\dim\MM_1$ to be the odd dimension of $\MM$.
	
	\subsection{Quasireductive supergroups} Recall that a linear algebraic supergroup $G$ is called quasireductive if the underlying algebraic group $G_0$ is reductive.  Let $\g:=\operatorname{Lie}G$ denote the Lie superalgebra of $G$.  Then by definition, a representation of $G$ (or a $G$-module) is a $\g$-module $V$ such that the action of $\g_{\ol{0}}$ on $V$ integrates to an action of $G_0$. We write $\operatorname{Rep}G$ for the category of all representations of $G$, and $\FF(G)$ for the subcategory of finite-dimensional $G$-modules.	
	
	For a quasireductive supergroup $G$, the category $\FF(G)$ is Frobenius i.e., there are enough projective objects and
	every projective object is injective, see \cite{S1}.  Since $\operatorname{Rep}G$ is the coinductive completion of $\FF(G)$, one may check that the class of projectives and the class of injectives coincide in $\operatorname{Rep}G$, and that they consist exactly of arbitrary direct sums of finite-dimensional indecomposable projectives (or equivalently injectives).
	
	\subsection{Homogeneous spaces and induction}
	Let $K$ be a quasireductive subgroup of $G$.  We may form the homogeneous supervariety $G/K$, and we
	write $\C[G/K]$ for its space of functions. Because $K$ and $G$ are quasireductive, $G/K$ in fact will be affine.  See \cite{MaT} for more on the constructions of quotients, and proof of the affinity of $G/K$.  Recall that Frobenius reciprocity tells us that $\Res_{K}^{G}$ is left adjoint to $\Ind_{K}^{G}$.

	\begin{lemma}\label{exactness_induction} If $K\subset G$ are quasireductive supergroups then the induction functor $\Ind^G_K:\operatorname{Rep}K\to \operatorname{Rep}G$ is exact and takes injectives to injectives. 
	\end{lemma}
	\begin{proof} By the geometric definition of induction, $\Ind^G_KM$ is isomorphic to the functor of global sections of the $G$-equivariant vector bundle
		$G\times_KM$. Exactness now follows from the affinity of $G/K$. By Frobenius reciprocity, injectivity of a $K$-module $M$ implies injectivity of $\Ind^G_KM$.  
	\end{proof}
	
	\begin{cor}\label{res}If $K\subset G$ are quasireductive supergroups, then the restriction functor $\Res_K:\operatorname{Rep}G\to\operatorname{Rep}K$ takes projectives to projectives.
	\end{cor}
\begin{proof}
	By Frobenius reciprocity, $\Res_K^{G}$ is left adjoint to $\Ind_{K}^{G}$.  
\end{proof}

		We need the following analogue of Shapiro's lemma for quasireductive supergroups.
	\begin{lemma}[Frobenius Reciprocity]\label{Fr} Let $K\subset G$ be quasireductive supergroups and $N$ be a $K$-module.
		For a $G$-module $M$, and for all $i$, we have natural isomorphisms
		\[
		\operatorname{Ext}^i_G(M,\Ind^G_KN)\cong\operatorname{Ext}^i_K(\Res_K^GM,N).
		\]
	\end{lemma}
	\begin{proof} If $I^{\bullet}$ is an injective resolution of $N$ then $\Ind^G_KI^{\bullet}$ is an injective resolution of $\Ind^G_KN$. The statement
		now follows from Frobenius reciprocity:
		$$\Hom_G(M, \Ind^G_KI^s)\simeq \Hom_K(\Res_K^GM, I^s).$$
	\end{proof}
	
	\subsection{Splitting subgroups}
	The $G$-module $\C[G/K]$ contains a trivial submodule $\C\sub\C[G/K]$ spanned by the constant function $1$.  In what follows, if $V$ is a $G$-module and $W\sub V$ is a $G$-submodule, we say $W$ splits off $V$ if $W$ is a direct summand of $V$ as a $G$-module. \color{black}
	
	\begin{prop}\label{splitting_subgroup_charac} Assume that $G$ is quasireductive and $K$ is a quasireductive closed subgroup.
		The following are equivalent:
		\begin{enumerate}
			\item The trivial $G$-submodule $\C$ splits off $\C[G/K]$ as a $G$-module;
			\item for any $G$-module $M'$ the natural injective morphism $M'\to \Ind^G_KM'$ splits;
			\item for any pair of $G$-modules $M,M'$, the restriction morphism
			\[
			\operatorname{Ext}^i_{G}(M,M')\to\operatorname{Ext}^i_K(M,M')
			\]
			is injective for all $i$.
		\end{enumerate}
	\end{prop}
	
	\begin{definition}
		We say that a subgroup $K\sub G$ is a splitting subgroup in $G$ if it is quasireductive and satisfies the above equivalent conditions.
	\end{definition}
	
	\begin{proof}
		(1)$\Rightarrow$(2) follows from a natural isomorphism $$\tau:\Ind^G_K M'\to (\Ind^G_K\C)\otimes M'.$$ Indeed, if $\pi:\Ind^G_K\C\to\C$ is a
		splitting homomorphism then $(\pi\otimes 1_{M'})\tau\varphi=1_{M'}$.
		
		Let us show that (2)$\Rightarrow$(3). Note that by (2) the natural homomorphism
		$$\varphi:\Ext_G^i(M,M')\to\Ext^i_G(M,\Ind^G_KM')$$ is an embedding. Using Lemma \ref{Fr} we have an isomorphism
		$$\psi:\Ext^i_G(M,\Ind^G_KM')\to \Ext^i_K(M,M').$$
		The composition $\psi\varphi$ is the restriction homomorphism.
		
		For $(3)\Rightarrow (1)$, we have a short exact sequence of $G$-modules
		\[
		0\to\C\to\C[G/K]\to\C[G/K]/\C\to 0.
		\]
		The evaluation morphism at $eK$ provides a splitting as $K$-modules, and thus by (3) the sequence must also split over $G$.
	\end{proof}
	
	\begin{cor}
		If $K$ is a splitting subgroup of $G$, then a $G$-module $M$ is projective if and only if its restriction to $K$ is projective.
	\end{cor}
	
	\subsection{Various sufficient criteria for splitting}In what follows we always assume that $G$, $H$, and $K$ are quasireductive.  Our first result follows immediately from the definition of splitting.

	\begin{cor}\label{cor_splitting_subgroup} Let $K\sub H\sub G$.
		\begin{enumerate} 	
			\item If $K$ is a splitting subgroup in $G$, then $H$ is also a splitting subgroup in $G$.
			\item If $K$ is a splitting subgroup of $H$, and $H$ is a splitting subgroup of $G$, then $K$ is a splitting subgroup of $G$.
		\end{enumerate}
	\end{cor}
	
	\begin{remark}\label{remark_induction}
		By Corollary \ref{cor_splitting_subgroup}, we may work inductively to prove that a given subgroup of $G$ is splitting, in the following sense: if $K\sub G$ is a subgroup then it is splitting if and only there exists a chain of inclusions $K=H_0\sub H_1\sub\cdots\sub H_k=G$ such that $H_i$ is a splitting subgroup of $H_{i+1}$ for all $i$.
	\end{remark}
		
	\begin{lemma}\label{lemma_splitting_normal_red}
		Suppose that $R$ is a normal subgroup of $G$ such that $R^\circ$ is reductive, and let $H=G/R$.  Write $\pi:G\to H$ for the quotient map.  Let $K\sub G$ be a subgroup, and let 
		\[
		\ol{K}:=\pi(K)=K/K\cap R\sub H.
		\]
		If $\ol{K}$ is a splitting subgroup of $H$, then $K$ is a splitting subgroup of $G$.  
	\end{lemma}
	
	\begin{proof}
		Write $\varphi:\C[H/\ol{K}]\to\C$ for the $H$-equivariant, and thus $G$-equivariant splitting ($H$ being a quotient of $G$).  Because $R$ has semisimple representation theory, we have a $G$-equivariant projection $\C[G/K]\to \C[G/K]^{R}=\C[H/\ol{K}]$.  Taking the composition with $\varphi$ gives our desired splitting.
	\end{proof}
	
	\begin{lemma}\label{lemma_splitting_normal}
		Suppose that $N\sub G$ is a normal subgroup of $G$, and write $H=G/N$.  Suppose that $K\sub G$ is a splitting subgroup.  Then $\ol{K}$ is a splitting subgroup of $H$, where $\ol{K}=K/K\cap N$.  
	\end{lemma}
	\begin{proof}
		We have a natural $G$-equivariant map $G/K\to H/\ol{K}$.  Thus we have an inclusion $\C[H/\ol{K}]\to\C[G/K]$, so composing with the splitting $\C[G/K]\to\C$ gives the result.
	\end{proof}
	
	\begin{lemma}\label{lem_split_quotient} Let $K\subset G$ be a subgroup such that $(\operatorname{Lie}G)_{\ol{1}}=(\operatorname{Lie}K)_{\ol{1}}$. Then $K$ is splitting in $G$.
	\end{lemma}
	\begin{proof} We have $G/K=G_0/K_0$. Hence the statement.
	\end{proof}
	
	For the proof of the following, see Section 6 of \cite{Sh2}.
	\begin{lemma}
		Suppose that $G$ is quasireductive, $K$ is a splitting subgroup of $G$, and $\h\sub\g_{\ol{1}}$ is an odd abelian ideal.  Then $\h\sub\k$.
	\end{lemma}
	
%
%
%

	\begin{lemma}\label{lemma_semidirect_prod}
		Suppose that $K\sub H\sub G$ are quasireductive supergroups such that $\g=\h\rtimes\l$ for a quasireductive Lie superalgebra $\l$.  Suppose further that $\l$ stabilizes $\k$. If $K\sub G$ is splitting, then $K\sub H$ is splitting as well.
	\end{lemma}

	\begin{proof}
			Applying Lemma \ref{lemma_splitting_normal}, we may pass to a finite cover of $G$ so that $G_0\cong H_0\times L_0$.  Then by the theory of Harish-Chandra pairs for supergroups, we obtain a decomposition $G=H\rtimes L$.  Since $\l$ stabilizes $\k$, $L$ must similarly stabilize $K$, and thus $K\rtimes L\sub G$ is a subgroup, and since it contains $K$ it is splitting.  Consider the natural $H$-equivariant map
			\[
			H/K\to (H\rtimes L)/(K\rtimes L).
			\]
			It is a surjective closed embedding, and since the supervarieties are smooth of the same dimension, it must be an isomorphism.  Since the trivial $G$-submodule splits off functions on the RHS as an $H$-module, it must do the same on the LHS, and we are done.
	\end{proof}

	\subsection{Geometric criterion}\label{section_odd_semisimple}
	Let $\g=\operatorname{Lie}G$.  An element $x\in\g_{\ol{1}}$ is called homological if $x^2:=\frac{1}{2}[x,x]$ is a semisimple element in $\g_{\ol{0}}$. For example, if $[x,x]=0$ then $x$ is homological.  We set
	\[
	\g_{\ol{1}}^{hom}:=\{x\in\g_{\ol{1}}: x \text{ is homological}\}.
	\]
	Observe that $\g_{\ol{1}}^{hom}$ is $G_0$-stable; however it need not be closed or open in $\g_{\ol{1}}$. 
	
	Let $x\in\g_{\ol{1}}^{hom}$.   Then if $M$ is a $G$-module, $[x,x]$ acts semisimply on $M$.  One can define a tensor functor
	$DS_x: \operatorname{Rep}G\to \operatorname{SVect}$ defined by:
	\[
	DS_xM:=\ker x/(xM\cap\ker x)=H(x,M^{[x,x]}).
	\]
	
	Assume that $\MM$ is a affine algebraic supervariety with a vector field $Q$ such that $Q^2$ acts semisimply on $\C[\MM]$.  We say that a point $p\in\MM_0$ is a zero of $Q$ if $Q(\C[\MM])\subset I_p$, where $I_p$ is the maximal ideal determined by $p$. Let $Z(Q)$ denote zero locus of $Q$, which is a Zariski-closed subset of $\MM_0$.
	
	\begin{lemma}\label{lem_orbits} If $Z(Q)=\emptyset$ then $DS_Q\C[\MM]=0$.
	\end{lemma}
	\begin{proof} Let $\JJ(\MM)$ denote the ideal generated by odd functions on $\MM$ and consider the  $\JJ(\MM)$-adic filtration
		on $\C[\MM]$.
		Consider the corresponding graded $\C[\MM_0]$-module  $\operatorname{gr}\C[\MM]$.  Note that it has a compatible structure of a
		$Q^2$-module since $\mathcal J(\MM)$ is $Q^2$-stable. The operator $Q:\C[\MM]\to\C[\MM]$ induces a vector field $\operatorname{gr}Q$
		in $\operatorname{gr}\C[\MM]$ of degree $-1$.   In particular, $\operatorname{gr}Q$ defines a morphism $\JJ(\MM)/\JJ^2(\MM)\to \C[\MM]/\JJ(\MM)$, which defines a section $\gamma$ of the normal bundle $\mathcal N_{\MM_0}$.
		
		The condition $Z(Q)=\emptyset$ implies that $\gamma$ has no zeros on $\MM_0$.  
		Therefore  $\gamma$ defines an injective morphism $\gamma:\mathcal L\to \mathcal N_{\MM_0}$ of a trivial line bundle  $\mathcal L$. 
		The dual morphism $\mathcal N^*_{\MM_0}\to\mathcal L^*$ splits by affinity of $\MM_0$. This implies the existence
		of $Q^2$-invariant $f\in \JJ(\MM)$ such that $\gr Q(\gr f)=1$. 
		
		We have thus shown that
		\[
		Q(f)= 1 +\eta
		\]
		where $\eta\in\mathcal{J}(\MM)$ is even, nilpotent, and $Q$-invariant.  Thus if we set $g=f/(1+\eta)$ we see that $Q(g)=1$, and we are done.  
	\end{proof}
	
	\begin{thm}\label{thm_orbits} Let $G$ be a quasireducitve supergroup and $K\subset G$ be a splitting subgroup. Then for any homological $x\in\g_{\ol{1}}$ the orbit $G_0x$ has non-trivial intersection with $\k_{\ol{1}}$; equivalently, $G_0\cdot\k_{\ol{1}}^{hom}=\g_{\ol{1}}^{hom}$.
	\end{thm}
	\begin{proof} Assume that there exists $x\in\g_{\ol{1}}$ such that $G_0x\cap \k=\emptyset$. Let $Q$ be the odd vector field on $G/K$ corresponding to $x$.
		Then the Lie algebra of the stabilizer of any $p\in (G/K)_0$ does not contain $x$ and hence $Z(Q)=\emptyset$. By Lemma \ref{lem_orbits} we have
		$DS_Q\C[G/K]=0$; but if $\C$ splits as a direct summand we must have $\C\subset DS_Q\C[G/K]$, giving a contradiction.
	\end{proof}

	\begin{cor}\label{cor_sdim_constraint}
	Suppose that $\g_{\ol{1}}^{hom}$ and $\k_{\ol{1}}^{hom}$ contain dense open subsets of $\g_{\ol{1}}$ and $\k_{\ol{1}}$, respectively.  Then if $K$ is splitting in $G$ we must have $\sdim G/K\geq0$.
\end{cor}

\begin{proof}
	Since $K$ is splitting in $G$, by Theorem \ref{thm_orbits}, $G_0\cdot\k_{\ol{1}}^{hom}=\g_{\ol{1}}^{hom}$.  By our assumptions, this implies that the image of the map $G_0\times\k_{\ol{1}}\to\g_{\ol{1}}$ contains an open subset of $\g_{\ol{1}}$.  Since we work over $\C$, it follows that for some $(g,x)\in G_0\times \k_{\ol{1}}$, the map of tangent spaces 
	\[
	T_{(g,x)}(G_0\times\k_{\ol{1}})\to T_{g\cdot x}\g_{\ol{1}}
	\]
	is surjective.  We may act by $g^{-1}$ so that WLOG $g=e$.  Then we have
	\[
	T_{(e,x)}(G_0\times\k_{\ol{1}})=\g_0\times\k_{\ol{1}}\to \g_{\ol{1}},
	\]	
	given by
	\[
	(u,x)\mapsto [u,x]+x.
	\] 
	Since $\k_{\ol{0}}$ preserves $\k_{\ol{1}}$, if we choose a splitting of vector spaces $\g_{\ol{0}}=\k_{\ol{0}}\oplus\m$, we obtain that the restriction of the above map to:
	\[
	\m\times \k_{\ol{1}}\to \g_{\ol{1}}
	\]
	remains surjective.  This implies that
	\[
	\dim\g_{\ol{0}}-\dim\k_{\ol{0}}+\dim\k_{\ol{1}}\geq \dim\g_{\ol{1}}.
	\]
	From this we obtain $\sdim G/K\geq 0$.
\end{proof}

	\section{Defect subgroups}
	\subsection{Contragredient superalgebras}  We call a finite-dimensional Lie superalgebra $\g$ contragredient if $\g\simeq\g(A)$ is
	a finite-dimensional Kac-Moody superalgebra
	with symmetrizable Cartan matrix. Then $\g=\bigoplus_{i=1}^s\g_i$ is a direct sum of Kac-Moody superalgebras
	with indecomposable Cartan matrix. Each $\g_i$ is either simple (basic classical or exceptional) or isomorphic to $\g\l(n|n)$ for
	$n\geq 1$. We call $\g$ indecomposable if $s=1$. Furthermore, $\g$ satisfies the following properties:
	\begin{enumerate}
		\item $\g_{\ol{0}}$ is reductive and $\g_{\ol{1}}$ is a semisimple $\g_{\ol{0}}$-module;
		\item the centralizer of the Cartan subalgebra $\h$ in $\g$ coincides with $\h$;
		\item in the root decomposition 
		\[
		\g=\h\oplus\bigoplus_{\alpha\in\Delta}\g_{\alpha},
		\] 
		all roots spaces $\g_{\alpha}$ have dimension $1|0$ or $0|1$;
		\item  $\g$ admits a nondegenerate, invariant supersymmetric form $(-,-)$, which we fix.
	\end{enumerate}
	We say that $\alpha\in\Delta$ is either even or odd if the corresponding root space is of that parity. In this section we denote by $G$
	a connected
	quasireductive algebraic supergroup with Lie superalgebra $\g$.
	
	\begin{lemma}\label{centralizer} For any $h\in\h$ the centralizer $\g^h$ is the direct product of a contragredient Lie superalgebras and a
	subalgebra of $\h$.
\end{lemma}
\begin{proof} Without loss of generality we may assume that $\alpha(h)\in\mathbb Q$ for all roots $\alpha$.
	We have the decomposition
	$$\g^h=\h\oplus\bigoplus_{\alpha\in\Delta(\g^h)}\g_{\alpha},\quad \Delta(\g^h)=\{\alpha\in\Delta\mid\alpha(h)=0\}.$$
	Let $\varepsilon$ denote the minimum of $\alpha(h)$ for all roots $\alpha$ such that $\alpha(h)>0$. Let $h_0\in\h$ be some
	regular element such that $|\alpha(h_0)|<1$ for $\alpha\in\Delta$.  Set $h'=h+\frac{\varepsilon}{2} h_0$ and let $\Delta^+$ be the set
	of positive roots defined by $h'$. If $B$ is the corresponding set of simple roots, then $B'=B\cap\Delta(\g^h)$ is the base
	of $\g^h$. By construction $B'$ is linearly independent. That implies $\g^h=\a\oplus \g'$ where $\a\subset Z(\g^h)$ is the subspace 
	complementary to $[\g^h,\g^h]\cap Z(\g^h)$. Then $\g'$ is a quotient of the Kac-Moody superalgebra by
	some central ideal $\c$. Since the restriction of the invariant form on $\g'$ is nondegenerate, $\c=0$.
	
\end{proof}

	\subsection{Defect subgroups}\label{section_def_subgroups}
	For an odd root $\alpha\in\Delta$, we say that $\alpha$ is isotropic if $(\alpha,\alpha)=0$, and otherwise we say $\alpha$ is non-isotropic. 
	
	The defect of $\g$ is the maximal number $d$ of mutually orthogonal linearly independent isotropic roots of $\g$.          
	
	For an odd isotropic root $\alpha$, choose non-zero elements $x\in\g_{\alpha}$ and $y\in\g_{-\alpha}$ so that $h:=[x,y]\neq 0$.  Then we obtain a subalgebra
	\[
	\s=\C\langle h,x,y\rangle\cong\s\l(1|1).
	\]
	We write $S$ for the corresponding subgroup of $G$; note that either $S\cong SL(1|1)$, or if $\alpha$ lies in a factor of $\g$ isomorphic to $\operatorname{Lie}D(2,1;a)$, then for irrational $a$, $h$ generates a two-dimensional torus in $G$.

	\begin{definition}
		A quasireductive subgroup $D\subset G$  is called a defect subgroup if its Lie superalgebra $\mathfrak d$ satisfies the following:
		\begin{enumerate}
			\item the $G_{0}$-orbit of any homological element of $\g_{\ol{1}}$ has a non-trivial intersection with $\mathfrak d_{\ol{1}}$;
			\item $\dim \mathfrak d_{\ol{1}}=2d$;
			\item $[\mathfrak d_{\ol{1}},\mathfrak d_{\ol{1}}]=\mathfrak d_{\ol 0}$, $[\mathfrak d_{\ol{1}},\mathfrak d_{\ol{0}}]=0$.
		\end{enumerate}
	\end{definition}
	\begin{prop}\label{prop_defect_emb} Let $\{\alpha_1,\dots,\alpha_d\}$ be a maximal set of mutually orthogonal linearly independent isotropic roots	in $\Delta$. Then the supergoup $D$ with Lie superalgebra $\mathfrak d$ generated by $\{\g_{\pm\alpha_i}\mid i=1,\dots d\}$ is a defect subgroup, and any defect subgroup is conjugate to $D$.
	\end{prop}
	\begin{proof} Let us prove the first assertion. Let $x$ be a homological odd element such that
		$[x,x]=h$. Without loss of generality we may assume that $h$ lies in the fixed Cartan subalgebra $\h$ of $\g$. Note that
		$x\in\g^h$ and that $\g^h$
		is a direct  sum of a contragredient superalgebra and a purely even center by Lemma \ref{centralizer}.
		Thus, if $h$ is not central, then we can reduce the statement to the contragredient superalgebra of
		smaller dimension and central $h$. Hence we can assume that $h$ is central and $\g$ is indecomposable.
		If the center of $\g$ is trivial, then $[x,x]=0$ and the statement follows from the description of the self-commuting cone in \cite{GHSS}. If $h\neq0$ then we must have 	$\g=\g\l(n|n)$, and one can check by direct computation that
		$x$ is conjugate to the matrix proportional to $\left(\begin{matrix}0&1_n\\ 1_n& 0\end{matrix}\right)$ which immediately
		implies the statement.

		Now we deal with the second assertion. 
		It is possible to choose $x_i\in\g_{\alpha_i}, y_i\in \g_{-\alpha_i}$ such that if we set $x=\sum_{i=1}^d (x_i+y_i)$
		then the Lie algebra of the minimal torus
		containing $h=[x,x]$ coincides with $\mathfrak d_{\ol{0}}$. Since $x$ is homological,
		using the action of $G_0$
		we can assume without loss of generality
		that the Lie superalgebra $\mathfrak d'$ of a defect subgroup $D'$ contains $x$ and hence $\mathfrak d_{\ol{0}}$. 
		Furthermore, 
		by direct inspection of root systems of all exceptional and basic classical superalgebras we learn that
		the centralizer of any generic element of $\mathfrak d_{\ol{0}}$
		is isomorphic to
		$\mathfrak d\oplus \l$ where $\l$ is a reductive Lie superalgebra. Therefore $\l_{\ol{1}}$ does not contain non-zero homological
		elements. Hence 
		$\mathfrak d'_{\ol{1}}\subset\mathfrak d_{\ol{1}}$ and by dimension count  $\mathfrak d=\mathfrak d'$.
	\end{proof}
	
	The main theorem of this text is:
	\begin{thm}\label{thm_S_is_splitting} Let $G$ be a quasireductive supergroup with contragredient Lie superalgebra $\g$ such that every
		direct summand of $\g$ has defect $\leq 1$ or is isomorphic to $\mathfrak{sl}(m|n)$ with $m\neq n$, or $\mathfrak{gl}(n|n)$.
		Then a defect subgroup is a splitting subgroup.  
	\end{thm}
	\begin{proof} If $K_1$ and $K_2$ are splitting in $G_1$ and $G_2$ respectively, then
		$K_1\times K_2$ is splitting in $G_1\times G_2$. Therefore it suffices to prove the statement for all simple Lie algebras of defect $1$ and for	$GL(m|n)$. The former is done in Section 4 and the latter in Section 6.
	\end{proof}

	\begin{cor}
		Let $S\sub GL(n|n)$ be a defect subgroup.  Then
		\begin{enumerate}
			\item $S$ lies in $SL(n|n)$ and is splitting;
			\item the image of $S$ in $PGL(n|n)$ and $PSL(n|n)$ is splitting.
		\end{enumerate}
	\end{cor}
	\begin{proof}
		The first statement follows from Lemma \ref{lemma_semidirect_prod}, and the second from Lemma \ref{lemma_splitting_normal}.
	\end{proof}

	For $Q(n)$ we have:
	
	\begin{thm}\label{thm_Q_splitting}
		The subgroups
		\[
		Q(2)^{n}\sub Q(2n), \ \ \ \ Q(2)^{n}\times Q(1)\sub Q(2n+1),
		\]
		are splitting.	
	\end{thm}

\begin{proof}
	Given in Section 6.
\end{proof}
	
	\begin{cor}
		Let $K\sub Q(n)$ be the splitting subgroup given in Theorem \ref{thm_Q_splitting}.  Then we have:
		\begin{enumerate}
			\item $K\cap SQ(n)$ is splitting in $SQ(n)$;
			\item the image of $K$ in $PQ(n)$ and $K\cap SQ(n)$ in $PSQ(n)$ are splitting.
		\end{enumerate}
	\end{cor}
	\begin{proof}
		For (1), consider the natural $SQ(n)$-equivariant map
		\[
		SQ(n)/(K\cap SQ(n))\to Q(n)/K.
		\]
		It is a surjective closed embedding between smooth supervarieties of the same dimension, and thus is an isomorphism.  Therefore since $\C$ is splits off $\C[Q(n)/K]$ over $SQ(n)$, the same must be true in $\C[SQ(n)/(K\cap SQ(n))]$.  
		
		Now (2) follows from Lemma \ref{lemma_splitting_normal}.
	\end{proof}

	\subsection{Density of $\g_{\ol{1}}^{hom}$} 
	\begin{lemma}\label{orbits} Let $\g$ be a contragredient Lie superalgebra.
		\begin{enumerate}
			\item For any $x\in \g_{\ol{1}}$ we have $\dim G_0\cdot x=\dim\g_{\ol{1}}-\dim (\g^x)_{\ol 1}$ where $\g^x$ denotes the centralizer of $x$ in $\g$.
			\item There exists a Zariski open set set $U\subset \mathfrak{d}_{\ol{1}}$ such that for all $x\in U$ 
			$$\operatorname{codim} G_0\cdot x=\dim (DS_x{\g})_{\ol{1}}+\dim (\mathfrak d^x)_{\ol{1}}.$$
		\end{enumerate}
	\end{lemma}
	\begin{proof} (1) follows from the fact that $\sdim G\cdot x=0$ since $G_0\cdot x$ has an odd symplectic form (see \cite{GHSS}).
		(2) For $x\in  \mathfrak{d}_{\ol{1}}$ set $y=[x,x]$. There exists a Zariski open dense set $U$ such that $\g^y=\h+\mathfrak d+ DS_x(\g)$
		for any $x\in U$. Then $DS_x(\g)\subset \g^x\subset \g^y$ implies
		$(\g^x)_{\ol{1}}=(DS_x{\g})_{\ol{1}}\oplus (\mathfrak d^x)_{\ol{1}}$.
	\end{proof}

	The following statement is proven, in different terms and perspective, in Theorem 3.3.1 of \cite{BoKN1}.

	\begin{prop}\label{density} Let $\g$ be a contragredient Lie superalgebra. 
		Then $\g^{hom}_{\ol {1}}$ contains a Zariski open subset in $\g_{\ol{1}}$ if and only if $DS_x(\g)_{\ol{1}}=0$ for some $x\in\g^{hom}_{\ol {1}}$.
	\end{prop}
	\begin{proof} Consider the action map $\varphi:G_0\times\mathfrak{d}_{\ol{1}}\to\g_{\ol{1}}$ defined by $(g,x)\mapsto gx$.
		We have to check when its image contains a dense
		open subset. It suffices to check that $$ D\varphi_{(e,x)}:\g_{\ol{0}}\oplus\mathfrak{d}_{\ol{1}}\to \g_{\ol{1}}$$ is surjective for some
		$x\in\mathfrak d_{\ol{1}}$. Moreover, we my assume that $x\in U$ as defined in (2) of Lemma \ref{orbits}.
		
		Let 
		$\tilde{\mathfrak d}:=\h+\mathfrak d$.
		Then the restriction of the form to $\tilde{\mathfrak d}$ is nondegenerate. Let us choose $x\in\mathfrak d_{\ol 1}$ as in the proof of
		Proposition \ref{prop_defect_emb}. Then
		$$[\h,x]=[\g_{\ol{0}},x]\cap \mathfrak{d}_{\ol 1}.$$
		Furthermore, the restriction of the form on $[\h,x]$ is non-degenerate and the orthogonal complement
		$[\h,x]^\perp\cap\mathfrak{d}_{\ol 1}$ coincides with $(\mathfrak d^x)_{\ol{1}}$.
		Therefore we obtain
		$$\mathfrak d_{\ol{1}}=([\g_{\ol{0}},x]\cap{\mathfrak d}_{\ol{1}})\oplus (\mathfrak d^x)_{\ol{1}}.$$
		By (2) of Lemma \ref{orbits},
		$$\dim [\g_{\ol{0}},x]=\dim G_0 \cdot x=\dim \g_{\ol{1}}-\dim (\mathfrak d^x)_{\ol{1}}-\dim (DS_x\g)_{\ol{1}}.$$
		The differential $D\varphi_{(e,x)}$ is given by the formula
		$D\varphi_{(e,x)}(u,y)=y+[u,x]$
		for all $u\in\g_{\ol{0}}$, $y\in\mathfrak d_{\ol{1}}$. Therefore
		$$\im D\varphi_{(e,x)}=[\g_{\ol{0}},x]+\mathfrak d_{\ol{1}}$$
		and
		$$\dim\im D\varphi_{(e,x)}=\dim [\g_{\ol{0}},x]+\dim\mathfrak d_{\ol{1}}-\dim ([\g_{\ol{0}},x]\cap\mathfrak d_{\ol{1}})=\dim[\g_{\ol{0}},x]+\dim(\mathfrak d^x)_{\ol{1}}.$$
		We have
		$$\dim\im D\varphi_{(e,x)}=\dim\g_{\ol{1}}-\dim (DS_x{\g})_{\ol{1}}.$$
		Thus, we obtain that $D\varphi_{(e,x)}$ is surjective if and only if $(DS_x\g)$ is an even Lie algebra.
	\end{proof}
	\begin{cor}\label{corollary_density_odd_iso_roots}
		If all odd roots of $\g$ are isotropic then $\g^{hom}_{\ol {1}}$ contains a Zariski open subset in $\g_{\ol{1}}$.
	\end{cor}

	Note that Corollary \ref{corollary_density_odd_iso_roots} is not `tight'; indeed, $\o\s\p(2m+1|2n)$ satisfies the hypotheses of Proposition \ref{density} whenever $m\geq n$, as does $\g(2|1)$. 

	\subsection{Minimal splitting subgroups} 
	By part (1) of Corollary \ref{cor_splitting_subgroup}, the condition of being a splitting subgroup is closed under taking supersets (i.e. sets containing a given one) in the poset of quasireductive subgroups of $G$.  Clearly this poset is stable under conjugation.	
	\begin{conj}\label{conj_min_splitting}
		If $G$ is a quasireductive supergroup with $\operatorname{Lie}G$ contragredient, then up to conjugacy a defect subgroup $D$ of $G$ is the unique minimal splitting subgroup.
	\end{conj}
	
	\begin{prop}
		Conjecture \ref{conj_min_splitting} holds for all $G$ of defect 1.
	\end{prop}
	
	\begin{proof}
		Let $K\sub G$ be a splitting subgroup; then necessarily $\k_{\ol{1}}$ contains elements $x,x'\in\g_{\ol{1}}^{hom}$ such that $h:=x^2$ and $h':=(x')^2$ are semisimple, nonzero, and non-conjugate.  We may conjugate so that there exists an odd isotropic root $\alpha$ such that $x\in\g_{\alpha}\oplus\g_{-\alpha}$; this is by Proposition \ref{prop_defect_emb} and the fact that $G$ is defect one.
		
		Choose a maximal torus $\mathfrak{t}$ of $\g$ containing $h$, and conjugate within $K$ (using reductivity of $K_0$) so that $h'$ also lies in $\mathfrak{t}$.  Now the centralizer of $\C\langle x,h\rangle$ in $\g$ is known to be $\C\langle h,y\rangle\oplus\g_x$, where $\g_x$ is a reductive Lie superalgebra (i.e. has $(\g_{x})_{\ol{1}}^{hom}=0$) and $y\in\g_{\alpha}\oplus\g_{-\alpha}$ has $[y,x]=0$ (see \cite{GHSS}).  If $\alpha(h')\neq0$, then we obtain $\g_{\alpha}$ and $\g_{-\alpha}$ inside $\k$, forcing the corresponding defect subgroup of $\alpha$ to lie in $K$, and we would be done.  Thus suppose instead that $\alpha(h')=0$, which implies that $h'\in\g_{x}\oplus\C\langle h\rangle$.  
		
		By the Jacobi identity we have $[x',[x,x']]=0$, which implies that $[x,x']\in\C\langle h'\rangle\oplus \g_{x'}$.  However since $\g_{x'}$ splits off $\g$ as a trivial module under $\langle x',h'\rangle,$ we must in fact have $[x,x']=rh'$ for some $r\in\C$.  If $r=0$, then $[x,x']=0$ which  implies that $x'\in(\C\langle h,y\rangle\times\g_{x})_{\ol{1}}^{hom}=\C\langle y\rangle$.  It follows that $x$ is a nonzero multiple of $y$, and since $\langle x,y'\rangle=\g_{\alpha}\oplus\g_{-\alpha}$ we are done. 
		
		Therefore suppose instead that $r\neq0$; then since $h'\in\g_{x}\oplus\C\langle h\rangle$ and $\g_x$ splits off $\g$ as a trivial module under $\ad(x)$, we must have $h'=sh$ for some $s\in\C^\times$.  
		
		From here, if we decompose $x'=\sum\limits x_{\beta}'$ into a sum of root vectors, for all $\beta\notin\{\pm\alpha\}$ we must have $\beta(h')\neq0$; on the other hand $[h',x']=0$, so it follows that $x'=x_{\alpha}'+x_{-\alpha}'$.  From here it is easy to see that $x,x'$ generate the defect subgroup determined by $\alpha$, and since this will lie in $K$ we are done.
	\end{proof}
	
	For $Q(n)$ we have the following clean result:
	
	\begin{prop}\label{splittingQ} Let $G=Q(n)$ and $K$ a splitting subgroup of $G$.  Then $K$ contains a subgroup conjugate to $Q(2)^{d}$ for $n=2d$ and to $Q(2)^d\times Q(1)$ for $n=2d+1$.
	\end{prop}
	\begin{proof} In the Lie superalgebra $\g=\q(n)$, the odd component is an adjoint $G_0=GL(n)$-module. Theorem \ref{thm_orbits} implies that without loss of generality we may assume that
		$\k_{\ol{1}}$ contains a diagonal matrix $u$ with distinct eigenvalues
		$\lambda_1,\dots,\lambda_n$ such that $\lambda^2_1,\dots,\lambda_n^2$ are linearly independent over $\mathbb Q$. Since $K_0\subset G_0$ is a reductive subgroup and $[u,u]\in\k_0$,
		this implies that $K_0$ contains a maximal torus of $G_0$.
		Hence $\k$ contains an even part $\h_0$ of the Cartan subalgebra $\h\subset \g$. Theorem \ref{thm_orbits} implies that $\k_{\ol{1}}$ contains, up to conjugacy, any $v\in\g_{\ol{1}}$ of full rank with eigenvalues $\mu_1,\dots,\mu_n$.
		Then $[v,v]$ is conjugate to some element $t\in\h_0$ which is also of full rank. Therefore there exists an odd element $t'\in\k$ of full rank such that $[t',t']=t$. If eigenvalues of $t$ are distinct then the
		condition $[t',t]=0$ implies $t'\in\h_{\ol{1}}$. That means $\k_{\ol{1}}\cap\h_{\ol{1}}$ contains a Zariski dense subset of $\h_{\ol{1}}$, hence $\h\sub\k$. Now we have that
		$$\k=\h\oplus\bigoplus_{\alpha\in\Delta_{\k}}\g_{\alpha}$$
		for some subset $\Delta_{\k}$ of roots. Since $\k_0$ is a reductive root subalgebra of $\g_0=\mathfrak{gl}(n)$, we get $\k_0=\mathfrak{gl}(n_1)\oplus\dots\mathfrak{gl}(n_p)$ for some $n_1+\dots+n_p=n$. Hence
		$\k=\q(n_1)\oplus\dots\oplus\q(n_p)$. The maximal rank of any $u\in k_{\ol{1}}$ such that $[u,u]=0$ equals
		$2\sum\lfloor n_i/2\rfloor$. But the maximal rank of a self-commuting odd element in $\q(n)$ is $2\lfloor n/2 \rfloor$. Since $\k_{\ol{1}}$ must contain a self-commuting element of maximal rank we obtain that for even $n$
		all $n_i$ are even and
		for odd $n$ at most one $n_i$ is odd. Therefore $K$ contains  $Q(2)^{d}$ for $n=2d$ and  $Q(2)^d\times Q(1)$ for $n=2d+1$.
	\end{proof}

	\section{Defect one cases: splitting for symmetric pairs}
	
	In this section we prove Theorem \ref{thm_S_is_splitting} for all defect one supergroups.  We do this via the application of Corollary \ref{cor_splitting_subgroup} to the following symmetric subgroups, which we will show are splitting:
	\begin{enumerate}
		\item $	GL(m|n-1)\times GL(1)\sub GL(m|n), \ \ n>m$;
		\item $	SOSp(m-1|2n)\sub SOSp(m|2n), \ \ m\text{ odd, or } \ m>2n$;
		\item $	SOSp(2m|2n-2)\times Sp(2)\sub SOSp(2m|2n), \ \ n>m$;
		\item $	SOSp(2|2)\times SO(2)\sub D(2,1;\alpha)$;
		\item $	D(2,1;3)\sub G(1|2)$;
		\item $	D(1,2;2)\times SL(2)\sub F(1|3)$.
	\end{enumerate}
	Using the above list, we reduce everything down to the case when $G=GL(1|1)$, for which the statement is obvious, or $G=SOSP(2|2)$; however $SOSP(2|2)\cong SL(1|2)$, so in fact this is enough.  
	
	We will show the above subgroups are splitting in different ways. 	For case (1) we refer to Theorem \ref{main_splitting}, and for case (2) we refer to Section 10 of \cite{Sh1}.   Case (4) will be dealt with separately in Section \ref{section_d(2,1;alpha)}.
	
	For cases (3), (5) and (6), we will first prove that:
	
	\begin{lemma}\label{lemma on dom wts}
		For $K\sub G$ listed in (3), (5), or (6), there exists a Borel subgroup $B\sub G$ such that if there exists a $B$-eigenvector $v$ of $\C[G/K]$ of weight $\lambda$, then either $\lambda=0$ or $\lambda$ does not lie in the principal block.
	\end{lemma}
	\begin{proof}
		See Section \ref{section_casimir}.
	\end{proof}

	\begin{prop}
		For $K\sub G$ listed in (3), (5), (6), $\C$ splits off $\C[G/K]$.
	\end{prop}
	\begin{proof}
		Write $W$ for the summand of $\C[G/K]$ in the principal block.  Then by Lemma \ref{lemma on dom wts}, $W$ must be indecomposable with socle $\C$.  If $\C$ does not split off, we may choose a submodule $V$ of $W$ given by a nonsplit short exact sequence $0\to\C\to V\to L\to 0$.  But then $V$ must be a highest weight module because for all groups in cases (3) (5), and (6), $0$ is a minimal dominant weight. But by assumption the highest weight of $V$ will not lie in $\a^*$, a contradiction.  Therefore $\C$ must split off.
	\end{proof}

	\subsection{Action of the Casimir}\label{section_casimir}  For cases (3), (5), and (6) we will show that the action of the Casimir is enough for proving that for $\lambda\neq0$ as in Lemma \ref{lemma on dom wts}, $\lambda$ does not lie in the principal block, thus proving Lemma \ref{lemma on dom wts}.
	
	Let $G/K$ be a supersymmetric space and write $(\g,\k)$ for the corresponding supersymmetric pair from an involution $\theta$ of $\g$.  We have a decomposition $\g=\k\oplus\p$, where $\p$ is the $(-1)$-eigenspace of $\theta$.  Choose a Cartan subspace $\a\sub\p$.  We may decompose $\g$ into the weight spaces of the action of $\a$, and we write $\ol{\Delta}$ for the nonzero weights of this action, which are exactly the restrictions of roots from $\h^*$ to $\a^*$ for a $\theta$-stable Cartan subalgebra $\h$ containing $\a$.  We refer to $\ol{\Delta}$ as the set of restricted roots for this supersymmetric pair.
	
	If we choose a positive set of restricted roots $\ol{\Delta}^+$, we may define
	\[
	\n:=\bigoplus\limits_{\alpha\in\ol{\Delta}^+}\g_{\alpha}.
	\]
	We say $(\g,\k)$ admits an Iwasawa decomposition if $\g=\k\oplus\a\oplus\n$ for some choice of $\n$ as above.   In this case we may choose a Borel subalgebra $\b$ containing $\a\oplus\n$.  We call such a Borel subalgebra an Iwasawa Borel subalgebra.
	
	\begin{lemma}
		Suppose that $(\g,\k)$ admits an Iwasawa decomposition $\g=\k\oplus\a\oplus\n$, and let $\b\supseteq\a\oplus\n$ be an Iwasawa Borel subalgebra.
		\begin{enumerate}
			\item The space of $\b$-eigenfunctions in $\C[G/K]$ of a given weight is at most one-dimensional, and non-zero $\b$-eigenfunction must be non-vanishing at $eK$.
			\item If there is a $\b$-eigenfunction of weight $\lambda$ appearing in $\C[G/K]$, then $\lambda\in\a^*$.
		\end{enumerate} 
	\end{lemma}
	\begin{proof}
		Because $\b+\k=\g$, $B$ has an open orbit on $G/K$, implying the first statement.  For the second statement, if $f\in\C[G/K]$ is a $\b$-eigenfunction of weight $\lambda\in\h^*$, then for $t\in\h\cap\k$ we have $0=(tf)(eK)=\lambda(t)f(eK)$.  On the other hand we have $\h=\t\oplus\a$, so this forces $\lambda\in\a^*$.
	\end{proof}

	\begin{lemma}\label{positivity} Assume that $(\g,\k)$ admits an Iwasawa decomposition such that
		\begin{enumerate}
			\item The restricted root system $\ol{\Delta}$ is an indecomposable even root system;
			\item Invariant form is definite on the root lattice $\mathbb Z\ol{\Delta}\subset\a^*$;
			\item If $m_\alpha=\sdim(\g_{\alpha})$, $\alpha_1\dots\alpha_k\in \ol{\Delta}^+$ are simple roots, $\rho=\frac{1}{2}\sum_{\alpha\in\ol{\Delta}^+} m_{\alpha}\alpha$, then $\rho=\sum c_i\alpha_i$ with $c_i\geq 0$ for all $i$;
			\item For every simple root $\alpha_i$ the root space $\g_{\alpha_i}$ has a non-trivial even part.
		\end{enumerate}
		Then the Casimir element $C$ acts with positive eigenvalue on any $\b$-eigenfunction $f_\lambda$ of weight $\lambda\neq 0$.
	\end{lemma}
	\begin{proof} One can compute that $C f_\lambda=(\lambda+2\rho,\lambda)f_\lambda$. The condition (4) implies the existence of a root $\mathfrak{sl}(2)$ subalgebra associated with every simple
		root $\alpha_i$. Therefore the dominance condition on $\lambda$ implies $(\lambda,\alpha_i)\geq 0$. Therefore we have
		\[
		(\lambda+2\rho,\lambda)=(\lambda,\lambda)+2\sum c_i(\lambda,\alpha_i)>0.
		\]
	\end{proof}

We now use Lemma \ref{positivity} to prove Lemma \ref{lemma on dom wts}. 
	\begin{itemize}
		\item $(\o\s\p(2m|2n),\o\s\p(2m|2n-2)\times\s\p(2))$, $n>m$: here $\a$ is one-dimensional and our restricted root system is $BC_1$.  If we write $\alpha$ for the positive long root, we obtain that
		\[
		\rho=(n-m-1)\alpha,
		\]
		which is a non-negative multiple of $\alpha$ by assumption.  
				
		\item $(\g(1|2),D(1,2;3))$:  Here $\a$ is given by the Cartan subalgebra of $\g_2$, and thus the form is definite on $\a^*$.  If we write $\alpha_1,\alpha_2$ for the simple roots of $\g_2$ where
		\[
		\frac{(\alpha_1,\alpha_1)}{(\alpha_2,\alpha_2)}=3,
		\]
		we obtain that
		\[
		\rho=\alpha_1+\alpha_2.
		\]
		\item $(F(3|1),D(1,2;2)\times\s\l(2))$: Here the Cartan subspace is given by the Cartan subalgebra of $\s\o(7)$, and thus the form is definite on $\a^*$.  If we write $\alpha_1,\alpha_2,\alpha_3$ for the simple roots of the restricted root system where
		\[
		\frac{(\alpha_1,\alpha_1)}{(\alpha_2,\alpha_2)}=3/8, \ \ \ \frac{(\alpha_2,\alpha_2)}{(\alpha_3,\alpha_3)}=2,
		\]
		we obtain that
		\[
		\rho=\alpha_1+2\alpha_2+3\alpha_3.
		\]
	\end{itemize}
	
	\subsection{$SOSp(2|2)\times SO(2)\sub D(2,1;\alpha)$}\label{section_d(2,1;alpha)}
	
	For this subgroup we will need to study the supersymmetric pair $(D(2,1;\alpha),\o\s\p(2|2)\times\s\o(2))$ more closely.  Present the root system of $D(2,1;\alpha)$ with basis $\epsilon_1,\epsilon_2,\epsilon_3$, even roots $\pm2\epsilon_i$, and odd roots $\pm\epsilon_1\pm\epsilon_2\pm\epsilon_3$.  Then for this supersymmetric pair we have $\a^*$ spanned by $\epsilon_1,\epsilon_2$, and we may take as simple root system $\epsilon_1-\epsilon_2-\epsilon_3,2\epsilon_2,2\epsilon_3$.  
	
	Then by \cite{Ger}, the dominant weights in the principal block, with respect to this positive system, are those of the form $\lambda_l=(l+1)\epsilon_1+(l-1)(\epsilon_2+\epsilon_3)$ for $l\in\Z^+$, and $\lambda_0=(0,0,0)$.  Thus the only $\lambda_l$ which lie in $\a^*$ are $\lambda_0$ and $\lambda_1$; further, by \cite{Ger}, the extension graph of the principal block is $D_{\infty}$, and looks as follows:
	
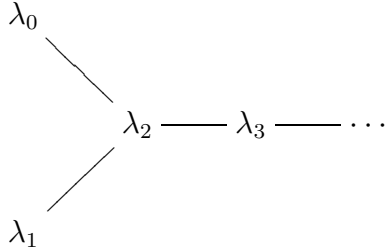
\[
	\xymatrix{ \lambda_0 \ar@{-}[dr]&  & & \\
	 & \lambda_2\ar@{-}[r] & \lambda_3\ar@{-}[r] & \cdots\\
 \lambda_1 \ar@{-}[ur] & & & }
	\] 
	In particular we have $\Ext^1(L(\lambda_1),L(\lambda_0))=0$.  
	
	Now let $W$ denote the summand of $\C[G/K]$ with trivial central character.  Suppose that $W$ is not semisimple; then consider its simple composition factors; from the extension graph we see that necessarily $L(\lambda_2)$ must appear.  However this means that the highest weight of $W$ is at least $\lambda_2$, which is larger than both $\lambda_1$ and $\lambda_0$.  Since $\lambda_i\notin\a^*$ for $i>1$, we obtain a contradiction.  Thus instead $W$ must be semisimple.
	
	\section {Volumes of supergrassmannians}
	
	\subsection{Invariants of representations of $\q(1)$}\label{section_5_alpha} In this section the ground field is $\mathbb R$.
	Consider a Lie superalgebra $\q$, defined over $\mathbb R$, with basis $\{Q,Q^2\}$ where $Q$ is an odd element.
	\begin{lemma}\label{linalgebra1} Let $V$ be a $\q$-module such that $Q:V\to V$ is an isomorphism and $Q^2:V\to V$ is a compact linear operator (the closure of the subgroup generated by $\operatorname{exp}Q^2$ is compact in $GL(V_{\ol{0}})\times GL(V_{\ol{1}})$).  Then the following hold:
		\begin{enumerate}
			\item $\dim V_{\ol{0}}=\dim V_{\ol{1}}=2n$ for some positive integer $n$;
			\item Every $Q^2$-invariant nondegenerate symmetric bilinear form $B_0$ on $V_{\ol{0}}$ extends uniquely to a bilinear supersymmetric $\q$-invariant form $B$ on $V$.        
		\end{enumerate}
	\end{lemma}
	\begin{proof}
		(1) follows from the fact that $Q^2$ defines a compact isomorphism on both $V_{\ol{0}}$ and $V_{\ol{1}}$.  For (2), observe that given a $Q^2$-invariant nondenegerate supersymmetric bilinear form $B_0$ on $V_{\ol{0}}$, for $v,w\in V_{\ol{0}}$ we must have $B(Qv,Qw)=B(Q^2v,w)$.   This indeed defines a form with the desired properties, and uniqueness also follows.
	\end{proof}
	Let us fix a volume form $\omega$ on $V$ and orientation $\omega_{\ol{0}}$ on $V_{\ol{0}}$. Let $B$ be a $\q$-invariant even symmetric bilinear form.
	Choose bases $e_1,\dots,e_{2n}\in V_{\ol{0}}$ and $f_1,\dots,f_{2n}\in V_{\ol{1}}$ such that
	\[
	\omega (e_1,\dots,e_{2n},f_1,\dots,f_{2n})=1, \omega_{\ol{0}}(e_1,\dots,e_n)>0.
	\]
	Set 
	\[
	\alpha(V,B,\omega):=(-\mathbf{i})^n\exp\left(\frac{\pi \mathbf{i} (\dim V_{\ol{0}}^+-\dim V_{\ol{0}}^-)}{4}\right)\frac{\operatorname{Pf}(B_1)}{\sqrt{|\det B_0|}},
	\]
	where $\dim V_{\ol{0}}^+-\dim V_{\ol{0}}^-$ is the signature of $B_0$, and $\operatorname{Pf}(B_1)$ is computed with respect to our chosen basis.
	
	\begin{lemma}\label{linalgebra2} If $V$ satisfies the conditions of Lemma \ref{linalgebra1} then $\alpha(V,B,\omega)$ does not depend on a choice of $B$.
	\end{lemma}
	\begin{proof} Let $B$ be some invariant form such that $B_0$ is positive definite and let $B'$ be some other invariant form. Then there is a $B_0$-orthonormal  basis $v_1,\dots,v_{2n}\in V_{\ol{0}}$ which is also orthogonal with respect to $B'_0$.
		Since $B'-tB$ is $\g$-invariant for any $t\in \mathbb R$ there is a decomposition $V=V^1\oplus\dots\oplus V^s$ orthogonal with respect to both forms $B$ and $B'$, where $V^i=\ker (B'-t_iB)$ where $t_1,\dots,t_s$ are eigenvalues of
		$B'_0$. Since each $V^i$ satisfies the assumption of Lemma \ref{linalgebra1} we have $\dim V^i=(2n_i|2n_i)$ with $n_1+\dots+n_s=n$. Let $B^i$ denote the restriction of $B$ to $V^i$ and $\omega^i$ a volume form on $V^i$ such that
		$\omega=\prod_{i=1}^s\omega^i$. Then by direct inspection $\alpha(V^i,B^i,\omega^i)=\alpha(V^i,t_iB^i,\omega^i)$ and therefore
		$$\alpha(V,B,\omega)=\prod_{i=1}^s\alpha(V^i,B^i,\omega^i)=\prod_{i=1}^s\alpha(V^i,t_iB^i,\omega^i)=\alpha(V,B',\omega).$$
	\end{proof}
	Since the coefficient $\alpha(V,B,\omega)$ does not depend on the choice of $B$ we just write $\alpha(V,\omega)$. The following lemma gives a practical way to compute $\alpha(V,\omega)$.
	\begin{lemma}\label{linalgebra3} Let $V$ satisfy the conditions of Lemma \ref{linalgebra1}. Let $e_1,\dots,e_{2n}\in V_{\ol{0}}$ be a positively oriented basis such that the matrix of $Q_0^2$ is skew-symmetric in this basis.
		(Such basis exists by compactness of $Q^2$.)
		For any basis $f_1,\dots,f_{2n}$ of $V_{\ol{1}}$ such that $\omega (e_1,\dots,e_{2n},f_1,\dots,f_{2n})=1$, the matrix  $\left(\begin{matrix}0&Q_{01}\\ Q_{10}&0\end{matrix}\right )$ of $Q$ satisfies the property
		$Q^t_{01}Q_{10}$ is skew-symmetric and $$\alpha(V,\omega)=\operatorname{Pf}(Q^t_{01}Q_{10}^{-1})$$
	\end{lemma}
	\begin{proof} Choose an invariant form $B$ on $V$ such that $e_1,\dots,e_{2n}$ is an orthonormal basis with respect to $B_0$; then by skew symmetry of $Q_0^2$, $B_0$ will be $Q^2$-invariant, so may be extended uniquely to an invariant form $B$ on $V$ using Lemma \ref{linalgebra1}. For any $v,w\in V_{\ol{1}}$ we have
		\[
		B(v,w)=B(v,QQ^{-1}w)=B(Qv,Q^{-1}w).
		\]
		Therefore $B_1=Q^t_{01}Q_{10}^{-1}$. That implies that $B_1$ is skew-symmetric. Using the definition of $\alpha(V,B,\omega)$ and noticing that $B_0$ is the identity matrix we get the desired result. 
	\end{proof}
	
	\begin{cor}\label{complex} Let $V=\mathbb C^{n|n}$ with standard complex basis $u_1,\dots,u_n\in V_0, v_1,\dots,v_n\in V_1$, and real volume form
		\[
		\omega=\prod_{i=1}^n\frac{dz_id\bar z_i}{d\zeta_id\bar\zeta_i}.
		\]
		If the action of $Q$ in the standard basis is given by
		\[
		Qu_i=(1+\mathbf{i})d_iv_i,\quad Qv_i=(1+\mathbf{i})c_iv_i
		\]
		for $c_i,d_i\in\R$,	then $\alpha(V,\omega)=\prod_{i=1}^n\frac{c_i}{d_i}$.
	\end{cor}
	\begin{proof} With respect to our complex basis,
		\[
		\ol{Q}^t_{01}Q_{10}^{-1}=-\operatorname{diag}(\frac{c_1}{d_1}\mathbf i,\dots, \frac{c_n}{d_n}\mathbf i).
		\]
			Here we recall that the transpose of a $n\times n$ complex matrix $A$ viewed as a real $2n\times 2n$-matrix is $\ol{A}^t$.  From here one easily checks that the Pfaffian of the corresponding real matrix is as desired.
	\end{proof}

	\subsection {Schwarz--Zaboronsky formula}  This section will deal with Berezinian integration on real supermanifolds.  We refer to \cite{M} for details.	
	
	Let $\mathcal M$ be a compact real supermanifold with underlying manifold $\MM_0$ and a volume form $\omega$. We fix an orientation on the underlying manifold $\MM_0$. Let $Q$ be a vector field on $\MM$.
	A point $p\in \MM$ is a zero of $Q$ if for any smooth function $f\in\mathcal C^{\infty}\MM$ we have $Q(f)\subset I_p$ where $I_p$ is the vanishing ideal of $p$. By $Z(Q)$ we denote the set of all zeros of $Q$.
	If $p\in Z(Q)$ then  $Q$ induces a linear operator on $T^*_p\MM=I_p/I_p^2$.
	We say that $p\in Z(Q)$ is isolated if $Q:T^*_p\MM\to T^*_p\MM$ is an isomorphism. Let us consider an odd vector field $Q$ on $\MM$ satisfying the following properties:
	\begin{enumerate}
		\item $Z(Q)$ is finite and consists of isolated zeros;
		\item $Q\omega=0$;
		\item $Q^2$ is a compact vector field, i.e.,there exists a compact Lie group $\mathcal K$ acting on $\MM$ such that $Q\in\operatorname{Lie}\mathcal K$.   
	\end{enumerate}
	Let us note that by Lemma \ref{linalgebra1} $\dim \MM=(2n|2n)$ for some $n\in\N$.   From \cite{SZ} we have:
	
	\begin{thm}\label{SZ} Let $\MM$ be a supermanifold of dimension $(2n|2n)$ with volume form $\omega$ and odd vector field $Q$ satisfy the above assumptions. Then there exists an odd function $\sigma$ such that 
		\begin{enumerate}
			\item $Q^2\sigma=0$;
			\item $Q\sigma\notin I_p$ for any $p\notin Z(Q)$; and
			\item  the Hessian of $Q\sigma$ at each point $p\in Z(Q)$ is nondegenerate. 
		\end{enumerate}
	Furthermore,
		$$           \int_{\MM}\omega=(2\pi)^{n}\sum_{p\in Z(Q)}\alpha(T_p\MM,\Hess_p(Q\sigma),\omega_p),$$
		where $\Hess_p(Q\sigma)$ is the Hessian of $Q\sigma$ at $p$, which will be an even symmetric $Q$-invariant form on $T_p\MM$.
	\end{thm}
	It follows from Lemma \ref{linalgebra2} that the function $\sigma$ can be removed from the Schwarz--Zaboronsky formula:
	\begin{cor}\label{simp}
		\[   
		\int_{\MM}\omega=(2\pi)^{n}\sum_{p\in Z(Q)}\alpha(T_p\MM,\omega_p).
		\]
	\end{cor}
	\begin{cor}\label{zero} If $Z(Q)=\emptyset$ then $ \int_{\MM}\omega=0$.
	\end{cor}
	\subsection{Supergrassmannian}
	The supergrassmannian $Gr(r|s, m|n)$ is the superscheme representing the functor of $(r|s)$-dimensional subspaces in $\mathbb C^{m|n}$.  We refer to \cite{M} and \cite{PS} for its construction and properties.  
	
	If $G=GL(m|n)$,	$K=GL(r|s)\times GL(m-r|n-s)$ and $P\subset G$ a maximal parabolic subgroup of $G$ containing $K$ then the homogeneous space $G/P$ is isomorphic to $Gr(r|s, m|n)$. The underlying algebraic variety is a product of two classical Grassmannians
	$Gr(r|m)\times Gr(s|n)$. If $p\in Gr(r|s, m|n)$ is the point with stabilizer $P$ then
	the tangent space $T_pGr(r|s, m|n)$ can be identified with $\mathfrak g/\mathfrak p$. This immediately implies the formulas for dimension and superdimension of $Gr(r|s, m|n)$:
	\begin{equation}\label{dim}
		\dim Gr(r|s, m|n)=(r(m-r)+s(n-s)|r(n-s)+s(m-r)),
	\end{equation}
	\begin{equation}\label{sdim}
		\operatorname{sdim} Gr(r|s, m|n)=(r-s)((m-r)-(n-s)).    
	\end{equation}
	Recall that $G$ has a compact real form $\mathcal U=U(m|n)$, the unitary supergroup. The Lie algebra
	$\mathfrak u(m|n)$ is the set of fixed points of antilinear involution $\theta$ defined by
	$$\theta\left(\begin{matrix}A&B\\ C&D\end{matrix}\right)=\left(\begin{matrix}-\bar A^t&\mathbf i\bar C^t\\ \mathbf i \bar B^t&-\bar D^t\end{matrix}\right).$$
	
	Let $\mathcal K:=\mathcal U\cap K=\mathcal U\cap P\simeq U(r|s)\times U(m-r|n-s)$.
	
	\begin{prop}\label{invariantform} The unitary supergroup $\mathcal U$ acts transitively on the supergrassmannian $Gr(r|s, m|n)$ and we have an isomorphism of real supermanifolds
		$$Gr(r|s, m|n)\simeq \mathcal U/\mathcal K.$$ Furthermore, $Gr(r|s, m|n)$ has a unique up to normalization $\mathcal U$-invariant volume form.
	\end{prop}
	\begin{proof} The first assertion follows from the corresponding classical fact for the underlying manifolds
		$$Gr(r|m)\times Gr(s|n)\simeq (U(m)/(U(r)\times U(m-r))\times (U(n)/(U(s)\times U(n-s))$$
		and comparison of odd dimensions of $\mathcal U/\mathcal K$ and $G/P$.
		
		To prove the second assertion we consider the point $p\in  Gr(r|s, m|n)$ which is stabilized by $\mathcal K$. The tangent space $T_p Gr(r|s, m|n)$ is a $\mathcal K$-module. One can identify it
		with the space $\Hom_{\mathbb C}(\mathbb C^{r|s},\mathbb C^{m-r|n-s})$. In this way a tangent vector is a $(r|s)\times (m-r|n-s)$-matrix $y$. The real part of the Hermitian form
		$H(y_1,y_2)=\operatorname{str}\theta(y_1)y_2$ defines a non-degenerate supersymmetric bilinear form $B$ on $T_p Gr(r|s, m|n)$ considered as a supervector space over $\mathbb R$. Since $\mathcal K$ is connected it is clear
		that $\mathcal K\subset SL(T_p Gr(r|s, m|n))$ and hence preserves a volume form $\omega_p$ on $T_p Gr(r|s, m|n)$. Then using a standard translation argument, $\omega_p$ induces a unique invariant form on $ Gr(r|s, m|n)$
	\end{proof}
	We choose an orientation $Gr(r|s, m|n)_0$ via its complex structure, i.e. so that for $z=x+iy$ the basis $dx,dy$ is positively oriented; similarly choose an orientation on $T_pGr(r|s, m|n)_{\ol{1}}$ according to its complex structure.
	
	If we let $\dim_{\R} Gr(r|s,m|n)=(d_1|d_2)$, then choosing local, positive oriented coordinates $u_1,\dots,u_{d_1},\xi_1,\dots,\xi_{d_2}$, we normalize $\omega$ so that in these local coordinates
	\[
	\omega=\sqrt{|\Ber B|}\frac{dx_1\cdots dx_{d_1}}{d\xi_1\cdots d\xi_{d_2}},
	\]
	where $B$ denotes the matrix form of the metric in our chosen local coordinate system.
	
	\begin{definition}
		With the above chosen volume form $\omega$ on $Gr(r|s,m|n)$, we set 
		\[
		V(r|s,m|n)=\int\limits_{Gr(r|s,m|n)}\omega.
		\]
	\end{definition}

	Observe that we have $V(r|s,m|n)=V(s|r,n|m)$.
	\subsection{Orientation subtlety and volume relationship}  
	
	Over the real supergroup $\UU(m|n)$ there is an obvious identification of homogeneous supervarieties $Gr(r|s,m|n)\cong Gr(m-r|n-s,m|n)$; however over $GL(m|n)$ these spaces are \emph{not} isomorphic.  Indeed the corresponding parabolics will not be conjugate in general.  However, the antilinear involution $\theta$ takes one parabolic to the other, inducing an isomorphism of real supervarieties 
	\[
	\Theta: Gr(r|s,m|n)\cong Gr(m-r|n-s,m|n)
	\]
	which is $\UU(m|n)$-equivariant.
	
	This subtle point is important because the orientation chosen on $Gr(r|s,m|n)$ comes from the complex structure induced as a $GL(m|n)$ homogeneous space.  Thus $\Theta$ will not in general take the chosen orientation on $Gr(r|s,m|n)$ to the chosen orientation on $Gr(m-r|n-s,m|n)$.  Since $\theta$ is acting by complex conjugate on coordinates, the orientation on the even part will be reversed $\dim_{\C} Gr(r|s,m|n)_0$ times.  On the other hand the invariant form $\omega$ will change by a sign of $(-1)^{\dim_{\C}Gr(r|s,m|n)_0+\dim_{\C}Gr(r|s,m|n)_1}$.  Therefore the volume will change by $(-1)^{\dim_{\C}Gr(r|s,m|n)_{1}}$, giving:
	\begin{lemma}\label{lemma_sign}
		\[
		V(r|s,m|n)=(-1)^{r(n-s)+s(m-r)}V(m-r|n-s,m|n). 
		\]
	\end{lemma}
	
	\subsection{Computation of volumes}
	In this section we seek to show that:
	\begin{thm}\label{grassmannmain} Let $\MM= Gr(r|s, m|n)$. Then the integral
		$\int_{\MM}\omega\neq 0$ if and only if $\sdim \MM\geq 0$.
	\end{thm}
	\begin{remark} Note that in the case $\sdim\MM=0$, the normalization of $\omega$ is independent of the choice of invariant form $B$ that we made.  
	\end{remark}

	\begin{prop}\label{negativecase} Assume that $\sdim \MM<0$. Then $\int_{\MM}\omega= 0$. 
	\end{prop}
	\begin{proof}	By Corollary \ref{zero} it suffices to construct an odd vector field $Q$ satisfying conditions (2) and (3) with $Z(Q)=\emptyset$.
         However here we have $\u_{\ol{1}}^{hom}=\u_{\ol 1}$ and $\mathfrak{v}_{\ol{1}}^{hom}=\mathfrak{v}_{\ol{1}}$, where $\mathfrak{v}=\operatorname{Lie}\KK$.  The argument of Corollary \ref{cor_sdim_constraint} still works in our situation, and thus we find that $\UU_0\cdot\mathfrak{v}_{\ol{1}}\neq\u_{\ol{1}}^{hom}$, and we are done.
	\end{proof}
	
Next we prove
	\begin{prop}\label{equalrank} Let $\MM= Gr(r|r, n|n)$. Then $\int_{\MM}\omega=(2\pi)^{2(n-r)r}\binom{n}{r}$.
	\end{prop}
	\begin{proof} Consider the standard basis $e_1,\dots,e_n,f_1,\dots,f_n$ in $\mathbb C^{n|n}$. Let $p\in\MM$ which corresponds to the subspace spanned by $e_1,\dots,e_r,f_1,\dots,f_r$.
		Choose $a_1,\dots,a_n\in\mathbb R^*$ such that $a_i\pm a_j\neq 0$ for any $i\neq j$. Let
		$$Q=\left(\begin{matrix}0&A\\ A&0\end{matrix}\right ),\quad A=(1+\mathbf i)\operatorname{diag}(a_1,\dots,a_n).$$
		Then $Q\in\operatorname{Lie}(\mathcal K)$. Moreover, it is clear that $Q:T_p\MM\to T_p\MM$ is an isomorphism.
		
		Let us compute $\alpha(T_p(\MM),\omega_p)$. We use the identification
		$T_p\MM\simeq \Hom_{\mathbb C}(\mathbb C^{r|r},\mathbb C^{n-r|n-r})$ and the $Q$-invariant decomposition
		\[
		T_p\MM=\bigoplus_{1\leq i\leq r,r<j\leq n}\Hom_{\mathbb C}(\mathbb Ce_i\oplus\mathbb Cf_i,\mathbb Ce_j\oplus\mathbb Cf_j).
		\]
		Define a basis $u_{ij}^{\pm},v_{ij}^{\pm}$ in $T_p\MM$ by
		$$u_{ij}^{\pm}=e_i\otimes e_j^*\pm f_i\otimes f_j^*,\quad v_{ij}^{\pm}=f_i\otimes e_j^*\pm e_i\otimes f_j^*.$$
		Then we have
		$$Qu_{ij}^{\pm}=(1+i)(a_i\mp a_j)v_{ji}^{\pm}, \quad Qv_{ij}^{\pm}=(1+i)(a_i\pm a_j)u_{ji}^{\pm}.$$
		Further $\theta u_{ij}^{\pm}=-u_{ji}^{\pm}$, $\theta v_{ij}^{\pm}=\pm iv_{ji}^{\pm}$.  Thus we have
		\[
		\langle u_{ij}^{\pm},u_{ij}^{\mp}\rangle=-2, \ \ \ \langle iu_{ij}^{\pm},iu_{ij}^{\mp}\rangle =2.
		\]
		\[
		\langle iv_{ij}^{\mp},v_{ij}^{\pm}\rangle=2.
		\]
		
		Let $z_{ij}^{\pm}=x_{ij}^{\pm}\pm iy_{ij}^{\pm}$ be the complex coordinate dual to $u_{ij}^{\pm}$, and let $\zeta_{ij}^{\pm}=\xi_{ij}^{\pm}+i\eta_{ij}^{\pm}$ be the complex coordinate dual to $v_{ij}^{\pm}$.  Then the bilinear form (as an element of $S^2(\TT^*\MM_p)$) at this points looks as follows:
		\[
		2\sum\limits_{i,j}\left(-dx_{ij}^+dx_{ij}^-+dy_{ij}^+dy_{ij}^-\right)-2\sum\limits_{i,j}\left(d\xi_{ij}^+d\eta_{ij}^-+d\xi_{ij}^-d\eta_{ij}^+\right).
		\]
		Thus the volume form at this point is given by:
		\[
		\frac{\prod\limits_{1\leq i\leq r<j\leq n}dx_{ij}^{+}dy_{ij}^{+}dx_{ij}^{-}dy_{ij}^{-}}{\prod\limits_{1\leq i\leq r<j\leq n}d\xi_{ij}^{+}d\eta_{ij}^{-}d\xi_{ij}^{+}d\eta_{ij}^{-}}=	\frac{\prod\limits_{1\leq i\leq r<j\leq n}dz_{ij}^{+}d\ol{z_{ij}^{+}}dz_{ij}^{-}d\ol{z_{ij}^{-}}}{\prod\limits_{1\leq i\leq r<j\leq n}d\zeta_{ij}^{+}d\ol{\zeta_{ij}^{+}}d\zeta_{ij}^{-}d\ol{\zeta_{ij}^{-}}}
		\]
	
		Therefore by Lemma \ref{complex} we find that:
		\[
		\alpha(T_p\MM,\omega_p)=\prod_{i\leq r,j>r}\frac{a_i+a_j}{a_i-a_j}\frac{a_i-a_j}{a_i+a_j}=1.
		\]
		In particular we see that $\alpha(T_p\MM,\omega_p)$ does not depend on a choice of $a_1,\dots,a_n$.
		
		Now consider $Q$ as a vector field on $\MM$ and let us compute $Z(Q)$. As in the proof of Proposition \ref{negativecase}
		we should describe
		\[
		S:=\{g\in\mathcal U_0\mid\operatorname{Ad}_g^{-1}(Q)\in\operatorname{Lie}\mathcal K\}.
		\]
		Then $Z(Q)$ is in bijection with $S/\mathcal K$. Note that the projection of $Q^2$ to either factor of $GL(n)\times GL(n)=GL(n|n)_0$ is a generic element of the Cartan subalgebra, so we can reduce the situation to
		computation elements $w\in W/W_K$ such that
		\begin{equation}\label{au1}
			w(Q)\in\operatorname{Lie}\mathcal K,
		\end{equation}
		where $W$ is the Weyl group of $\mathcal U$ and $W_K$ is the Weyl group of $\mathcal K$.
		Note that $W\simeq S_n\times S_n$. Let $w=(\sigma,\tau)$ for some $\sigma,\tau\in S_n$. Then $w$ satisfies (\ref{au1}) iff for each $i$ either $\sigma(i),\tau(i)\in \{1,\dots,r\}$ or
		$\sigma(i),\tau(i)\in\{r+1,\dots,n\}$. Therefore $w=(\sigma,\sigma)w'$ for some $w'\in W_K$. Thus, $|Z(Q)|=\binom{n}{r}$. Moreover, $Z(Q)$ is on the orbit of the diagonal $S_n$ in $W$.
		Furthermore,
		$$\alpha(T_{wp}\MM,\omega_{wp})=\alpha(T_p\MM,\omega_p)=1,$$
		since $w^{-1}Q$ has the same form as $Q$ with permuted $a_1,\dots,a_n$.
		Thus, Theorem \ref{SZ} gives $\int_{\MM}\omega=(2\pi)^{2r(n-r)}\binom{n}{r}$.
	\end{proof}

	\begin{prop}\label{one-zero} Let $m>n$.  Then we have
		\[
V(m-n|0,m|n))=(-1)^{n(m-n)}V(n|n,m|n)=(2\pi)^{n(m-n)}.		
		\]
	\end{prop}
	\begin{proof} We show that $V(m-n|0,m|n)=(2\pi)^{n(m-n)}$, with the second equality following from Lemma \ref{lemma_sign}.
		
		Choose $Q\in\operatorname{Lie}U(n|n)\subset\operatorname{Lie}{\mathcal K}$ in the same way as in the proof of Lemma \ref{equalrank}. One can easily see $Q$ has exactly one zero point $p$ with stabilizer
		$\mathcal K$; it remains to compute $\alpha(T_pGr(m-n|0,m|n),\omega_p)$.  In the first case, $T_pGr(m-n|0,m|n)=\Hom(\C^{m-n|0},\C^{n|n})$, and we may take as basis
		\[
		u_{ij}=e_i\otimes e_j^*, \ \ \ \ v_{ij}=f_i\otimes e_j^*, \ \ \ 1\leq i\leq n, \ n<j\leq m.
		\]
		We see that
	\[
	Qu_{ij}=(1+i)a_iv_{ij},\ \ \ \ Qv_{ij}=(1+i)a_iu_{ij}.
	\]
	Thus following the same reasoning as in Lemma \ref{equalrank}, we find that
	\[
	\alpha(T_p\MM,\omega_p)=\prod\limits_{i,j}\frac{a_{i}}{a_{i}}=1.
	\]
	\end{proof}
	In what follows, we set $V(a,b):=V(a|0,b|0)=V(0|a,0|b)$ for $0\leq a\leq b$, and use the notation $I(\MM,\omega):=\int_{\MM}\omega$.
	\begin{cor}\label{cor-one} Let $a\leq b\leq c$. Then
		$$V(b|a,c|a)=\frac{V(b-a,c-a)V(a|a,c|a)}{V(a|a,b|a)}.$$
	\end{cor}
	\begin{proof} Consider the partial flag variety $Fl(a|a,b|a,c|a)$ with invariant form $\tilde\omega$. We use the double bundle
		$$ Gr(b|a, c|a)\xleftarrow{Gr(a|a,b|a)}Fl(a|a,b|a,c|a)\xrightarrow{Gr(b-a,c-a)} Gr(a|a, c|a).$$
		The fibers are marked on the arrows. By $\mathcal U$-equivariance we can choose an invariant volume form $\tilde\omega$ on
		$Fl(a|a,b|a,c|a)$ such that
		$$\tilde\omega=\omega_f\omega_b=\omega'_f\omega'_b,$$
		where $\omega_f$ ( respectively, $\omega'_f$) are is a volume form on the fiber of the left (respectively, right) bundle and $\omega_b$ (respectively, $\omega'_b$) is a volume form on the base of the left
		(respectively, right) bundle.
		
		Using  the right bundle we get
		$$I({Fl(a|a,b|a,c|a)},\tilde\omega)=I(Gr(b-a,c-a),\omega'_f)I(Gr(a|a,c|a),\omega'_b).$$
		Applying now the left bundle we obtain
		$$I({Fl(a|a,b|a,c|a)},\tilde\omega)=I({Gr(a|a,b|a)},\omega_f)I({Gr(b|a,c|a)},\omega_b).$$
		That gives the desired formula.
	\end{proof}
	\begin{prop}\label{general_positive} Assume that $\sdim Gr(r|s,m|n)\geq 0$ and $r\geq s$ then
		$$V(r|s,m|n)=(-1)^{(r-s)(n-s)}\frac{V(s|s,n|n) V(n|n,m|n)V(r-s,m-n)}{V(s|s,r|s)V(n-s|n-s,m-r|n-s)}.$$
	\end{prop}
	\begin{proof} We start with the case $r=s$ and use again the same method as in the proof of Corollary \ref{cor-one} and we use the same notations.
		Note that $m>n$. Consider the partial flag variety $Fl(s|s,n|n,m|n)$ and the double bundle
		\[
		Gr(s|s,m|n)\xleftarrow{Gr(n-s|n-s,m-s|n-s)}Fl(s|s,n|n,m|n)\xrightarrow{Gr(s|s,n|n)} Gr(n|n,m|n).
		\]
		Then by Proposition \ref{equalrank} and Proposition \ref{one-zero} we obtain
		$$I({Fl(s|s,n|n,m|n)},\tilde\omega) =V(s|s,n|n)V(n|n,m|n)\neq 0.$$
		From the left bundle we get
		$$I({Fl(s|s,n|n,m|n)},\tilde\omega)=V(n-s|n-s,m-s|n-s)V(s|s,m|n).$$
		That gives $$V(s|s,m|n)=\frac{V(s|s,n|n)V(n|n,m|n)}{V(n-s|n-s,m-s|n-s)}.$$
		For the general statement consider the double bundle
		$$ Gr(r|s,m|n)\xleftarrow{Gr(s|s,r|s)}Fl(s|s,r|s,m|n)\xrightarrow{Gr(r-s|0,m-s|n-s)}Gr(s|s,m|n).$$
		Then 
		$$I({Fl(s|s,r|s,m|n)},\tilde\omega)=V(r-s|0,m-s|n-s)V(s|s,m|n).$$
		On the other hand, from the left bundle 
		$$I({Fl(s|s,r|s,m|n)},\tilde\omega)=V(s|s,r|s)V(r|s,m|n),$$
		and using Lemma \ref{lemma_sign} we get
		$$V(r|s,m|n)=(-1)^{(r-s)(n-s)}\frac{V(s|s,n|n) V(n|n,m|n)V(m-r|n-s,m-s|n-s)}{V(s|s,r|s)V(n-s|n-s,m-s|n-s)}.$$
		Using Corollary \ref{cor-one} for $a=n-s,b=m-r,c=m-s$ we have
		$$V(m-r|n-s,m-s|n-s)=\frac{V(r-s,m-n)V(n-s|n-s,m-s|n-s)}{V(n-s|n-s,m-r|n-s)},$$ and substitution gives the desired formula.
	\end{proof}

	\begin{cor}\label{formula} Assume that $r\geq s$  and $\operatorname{sdim}Gr(r|s,m|n)\geq 0$. Then
		$$\int_{Gr(r|s,m|n)}\omega=(-1)^{s(m+n+r+s)}\binom{n}{s}(2\pi)^A\int_{Gr(r-s,m-n)}\omega,$$
		where $A=(m-r)s+(n-s)r=\dim Gr(r|s,m|n)_1$.
	\end{cor}
	Now Corollary \ref{formula} and Proposition \ref{negativecase} imply Theorem \ref{grassmannmain}. 
	\subsection{The volume of $Q$-Grassmannian}
	In this subsection we assume that $G=Q(n)$ and $P$ is the parabolic subgroup with the Levi subgroup $K=Q(r)\times Q(n-r)$. The homogeneous space $QGr(r,n)=G/P$ is called a $Q$-Grassmannian. It represents the functor
	of all $r|r$-dimensional subspaces in $\mathbb C^{n|n}$ invariant under the fixed odd linear operator $F:\mathbb C^{n|n}\to\mathbb C^{n|n}$.
	As in the case of $GL$ the complex quasireductive group $Q(n)$ has a compact form $\mathcal U=UQ(n)$ and the Lie algebra of $\mathcal U$ is the set of fixed of the involution $\theta$ restricted to $\q(n)$ defined in the previous section.
	It is well defined as $\q(n)\subset\mathfrak{gl}(n|n)$ is $\theta$-stable. We set $\mathcal K=\mathcal U\cap K\simeq UQ(r)\times UQ(n-r)$. Here is an analogue of Proposition \ref{invariantform}.
	\begin{prop}\label{Q-compact} There is an isomorphism $QGr(r,n)\simeq G/P\simeq \mathcal U/\mathcal K$. Furthermore, there exists an unique up to normalization $\mathcal U$-invariant volume form $\omega$ on $\MM:=QGr(r,n)$.
	\end{prop}

	\begin{proof} The first assertion can be proven in the exact same way as Proposition \ref{invariantform}. To prove the existence of invariant form let us consider a point $p\in \MM$ with stabilizer $\mathcal K$. There is an
		obvious identification of the tangent space $T_p\MM$ with $\Hom_{\mathbb C[F]}(\mathbb C^{r|r},\mathbb C^{n-r|n-r})$. Note that in contrast with the case of undecorated supergrassmannians the $\mathcal K$-module
		$T_p\MM$ is not self-dual. However, since $\mathcal{K}$ admits no nontrivial characters, we have that $\mathcal K\subset SL(T_p\MM)$. It will be useful to choose a precise coordinate system in $T_p\MM$. If we represent a tangent vector by a matrix of the form
		$\left(\begin{matrix}Z&W\\W&Z\end{matrix}\right)$ where $Z,W$ are $r\times (n-r)$-matrices, then
		$$\omega_p=\prod_{i\leq r,j>r}\frac {dz_{ij}d\bar z_{ij}}{dw_{ij}d\bar w_{ij}}.$$
	\end{proof}

	\begin{prop}\label{Qgrassmann} Let $\MM=QGr(n,r)$. Then $\int_{\MM}\omega=(2\pi)^{2r(n-r)}C(r,n)$, where
		\[
	C(r,n)=\begin{cases}\binom{m}{l}\ \text{if}\ n=2m, r=2l, n=2m+1, r=2l+1,2l\\ 0\ \  \ \ \text{if}\ n=2m, r=2l+1\end{cases}.
		\]
		In particular, $\int_{\MM}\omega\neq 0$ if and only if $r(n-r)$ is even.
	\end{prop}
	\begin{proof} The idea of the proof is to use again the Schwarz--Zaboronsky formula for a suitably chosen odd vector field. We observe that $Q$ constructed in the proof of proposition
		\ref{equalrank} is an element of $\operatorname{Lie}UQ(n)$ via the embedding $UQ(n)\subset U(n|n)$. It is straightforward to check that $Z(Q)$ has $\binom{n}{r}$ points and all these points are obtained from the point
		$p$ by the action of the Weyl group $S_n$. Every such point represents the coordinate subspace of $\mathbb C^{n|n}$ with basis $u_{i_1},\dots,u_{i_r},v_{i_1},\dots,v_{i_r}$ if $u_1,\dots,u_n,v_1,\dots,v_n$ is the standard basis.
		Thus, we enumerate the points of $Z(Q)$ by $r$ element subsets $S$ of $\{1,\dots,n\}$ and use the notation $p_S$. Choosing exactly half of the basis $u_{ij}^+,v_{ij}^+$ defined in the proof Proposition  \ref{equalrank} we obtain
		$$\alpha(S):=\alpha(T_{p_{S}}\MM)=\prod_{i\in S,j\notin S}\frac{a_i+a_j}{a_i-a_j}.$$
		Thus, we have reduced the proof of the theorem to computation of the expression
		\[
		C(r,n):=\sum_{|S|=r}\alpha(S).
		\]
		First, note that $C(r,n)$ is a symmetric rational function in $a_1,\dots,a_n$ of homogeneous degree $0$ with denominator given by the Vandermonde determinant. That actually means that it is a constant which does not depend on
		$a_1,\dots,a_n$ as it should be by Theorem \ref{SZ}.
		Next, the substitution $a_n=0$ gives the recursions:
		\[
		C(r,n)=C(r,n-1)+(-1)^{n-r}C(r-1,n-1),\quad C(r,n)=(-1)^{r(n-r)}C(n-r,n), \quad C(1,2)=0.
		\]
		From this one may show that $C(r,n)$ has the desired values.
	\end{proof}
	\section{Splitting subgroups for $GL(m|n)$ and $Q(n)$}
	\begin{thm}\label{main_splitting}
		
		(a) Let $G=GL(m|n)$, $K=GL(r|s)\times GL(m-r|n-s)$. Then $K$ is a splitting subgroup of $G$ if and only if $\sdim Gr(r|s,m|n)\geq 0$.
		
		(b) Let $G=Q(n)$, $K=Q(r)\times Q(n-r)$. Then $K$ is a splitting subgroup of $G$ if and only if $r(n-r)$ is even.
	\end{thm}
	\begin{proof} If $\sdim Gr(r|s,m|n)<0$ or $r(n-r)$ is odd, the necessary condition for $K$ does not hold. Therefore it remains to show that if  $\sdim Gr(r|s,m|n)\geq 0$ or $r(n-r)$ is even the trivial $G$-submodule splits in
		$\mathbb C[G/K]$. We use unitary trick.
		Consider the embedding $\MM=\mathcal U/\mathcal K\to G/K$ as the closed $\mathcal U$-orbit of $p\in G/K$ with stabilizer $K$. Define the map $I:\mathbb C[G/K]\to\mathbb C$ by
		$$I(f):=\int_{\MM}f\omega.$$
		Note that $I$ is $\mathcal U$-invariant and hence $G$-invariant.
		Theorem \ref{grassmannmain} and Theorem \ref{Qgrassmann} imply $I(1)\neq 0$. Therefore $\mathbb C$ splits.
	\end{proof}
	\begin{cor}\label{defect-splitting}
		(a) Let $G=GL(m|n)$ with defect $d=\min(m,n)$, then the defect subgroup $D\simeq SL(1|1)^d$ is splitting in $G$ and hence any subgroup of $G$ containing $D$ is splitting.
		
		(b) Let $G=Q(2d)$ (respectively, $Q(2d+1)$) then the subgroup $Q(2)^{d}$ (respectively, $Q(2)^d\times Q(1)$) is splitting.
	\end{cor}
	\begin{proof} For (b) we just use Theorem \ref{main_splitting} and transitivity Corollary \ref{cor_splitting_subgroup}(2). For (a) assume that $d=n$, then by the same argument $GL(1|1)^d\times GL(m-n)$ is splitting.
		Furthermore, $H:=GL(1|1)^d$ is splitting in  $GL(1|1)^d\times GL(m-n)$. Therefore $H$ is splitting in $G$. Finally $D$ is splitting in $H$ by Lemma \ref{lem_split_quotient}. 
	\end{proof}
	
	\bibliographystyle{amsalpha}

\begin{thebibliography}{99999999}
		\bibitem[B]{B} D. Benson. {\em Modular representation theory: new trends and methods},  Springer Science $\&$ Business Media Vol. 1081, (1984).
		\bibitem[BoKN1]{BoKN1} B. Boe, J. Kujawa, B. Nakano. {\em Cohomology and support varieties for Lie superalgebras}, Transactions of the American Mathematical Society 362.12, 6551-6590 (2010).
		\bibitem[BoKN2]{BoKN2} B. Boe, J Kujawa, D. Nakano. {\em Cohomology and support varieties for Lie superalgebras II}, Proceedings of the London Mathematical Society 98.1, 19-44 (2009).
		\bibitem[GGNW]{GGNW}   D. Grantcharov, N. Grantcharov, D. Nakano, J. Wu. {\em On BBW parabolics for simple classical Lie superalgebras}, Advances in Mathematics 381, 107647 (2021).
		\bibitem[EAS]{EAS} I. Entova-Aizenbud, V. Serganova. {\it Jacobson-Morozov lemma for algebraic supergroups}, Advances in Mathematics 398, 108240, (2022).
		\bibitem[GHSS]{GHSS} M. Gorelik, C. Hoyt, V. Serganova, A. Sherman. {\it The Duflo-Serganova functor, vingt ans apres}, Journal of the Indian Institute of Science, 102, 961-1000 (2022).
		\bibitem[Ger]{Ger} J. Germoni. {\it Indecomposable representations of osp(3,2), D(2,1;$\alpha$), and $G(3)$}, Boletin de la Academia Nacional de Ciencias 65, 147-163, (2000).
		\bibitem[M]{M} Yu. I. Manin. {\it Gauge fields and complex geometry}, Moscow Izdatel Nauka (1984).
		\bibitem[MPV]{MPV} Yu. I. Manin, I. Penkov, A. A. Voronov. {\it Elements of supergeometry}, Journal of Soviet Mathematics 51.1, 2069-2083 (1990).
		\bibitem[MaT]{MaT} A. Masuoka, Y. Takahashi. {\it Geometric construction of quotients G/H in supersymmetry}, Transformation Groups 26.1, 347--375, (2021).
		\bibitem[PS]{PS} I.B. Penkov, I.A. Skornyakov. {\em Projectivity and affinity of flag supermanifolds}, Russian Mathematical Surveys 40.1, 233 (1985).
		\bibitem[SZ]{SZ} A. Schwarz, O. Zaboronsky. {\it Supersymmetry and localization}, Comm. Math. Phys. 183, 463--476 (1997).
		\bibitem[S1]{S1} V. Serganova {\it Quasireductive supergroups}, New Developments in Lie Theory and its Applications 544, 141--159 (2011).
		\bibitem[S2]{S2} V. Serganova {\it On generalizations of root systems}, Communications in Algebra 24.13, 4281-4299 (1996).
		\bibitem[Sh1]{Sh1} A. Sherman. {\it Spherical indecomposable representations of Lie superalgebras}, Journal of Algebra 547, 262-311 (2020).
		\bibitem[Sh2]{Sh2}A. Sherman. {\it Spherical supervarieties}, Annales de L'institut Fourier 71, 1449-1492 (2021).
		\bibitem[V]{V} Th. Voronov. {\it On volumes of classical supermanifolds}, Sbornik: Mathematics 207.11, 1512 (2016).		
		\bibitem[Z]{Z} V. Zakharevich. {\it Localization and stationary phase approximation on supermanifolds}, Journal of Mathematical Physics 58.8, 083506 (2017).
	\end{thebibliography}

\textsc{\footnotesize Vera Serganova, Dept. of Mathematics, University of California at Berkeley, Berkeley, CA 94720} 

\textit{\footnotesize Email address:} \texttt{\footnotesize serganov@math.berkeley.edu}

\textsc{\footnotesize Alexander Sherman, Dept. of Mathematics, Ben Gurion University, Beer-Sheva,	Israel} 

\textit{\footnotesize Email address:} \texttt{\footnotesize xandersherm@gmail.com}
\end{document}